\documentclass[12pt]{article}
\usepackage[latin1]{inputenc}

\usepackage{fullpage}
\usepackage{amsfonts}
\usepackage{amssymb}
\usepackage{amsmath}
\usepackage{color}
\usepackage[mathscr]{euscript}

\usepackage{mathtools}
\usepackage{stmaryrd}
\usepackage{wasysym}

\newtheorem{theorem}{Theorem}
\newtheorem{definition}[theorem]{Definition}
\newtheorem{lemma}[theorem]{Lemma}
\newtheorem{corollary}[theorem]{Corollary}
\newenvironment{proof}{\small{\bf Proof.}}%
  {\hfill$\Box$\normalsize\bigskip}

{\normalsize}

{\normalsize}

{\normalsize}

\newtheorem{proposition}[theorem]{Proposition}

\newcommand{\eqsepv}{\; , \enspace}       
\newcommand{\eqfinv}{\; ,}                
\newcommand{\eqfinp}{\; .}

\renewcommand{\bar}{\overline}

\newcommand{\mtext}[1]{\,\mbox{#1}\,} 
\newcommand{\module}[1]{| #1 |}
\newcommand{\norm}[1]{\|#1\|}

\newcommand{\sequence}[2]{\left\{#1\right\}_{#2}}           
\newcommand{\proscal}[2]{\left\langle#1\:,#2\right\rangle}  

\newcommand{\np}[1]{(#1)}                                   
\newcommand{\bp}[1]{\big(#1\big)}                           
\newcommand{\Bp}[1]{\Big(#1\Big)}                           
\newcommand{\bgp}[1]{\bigg(#1\bigg)}                        

\newcommand{\Bc}[1]{\Big[#1\Big]}                           

\newcommand{\na}[1]{\{#1\}}                                 
\newcommand{\ba}[1]{\big\{#1\big\}}                         
\newcommand{\Ba}[1]{\Big\{#1\Big\}}                         

\newcommand{\RR}{{\mathbb R}} 
\newcommand{\NN}{{\mathbb N}} 

\newcommand{\AAA}{{\mathbb A}} 
\newcommand{\BB}{{\mathbb B}} 

\newcommand{\defset}[2]{\left\{#1\:\left|\:#2\right.\right\}}

\newcommand{\uncertain}{w}
\newcommand{\Uncertain}{W}
\newcommand{\UNCERTAIN}{{\mathbb W}}

\def\stackops#1#2#3{%
  \mathrel{\vbox{\offinterlineskip\ialign{%
    \hfil##\hfil\cr
    $#1$\cr
    \noalign{\kern#3}
    $#2$\cr}}}}

\def\plusdot{\stackops{\cdot}{+}{-2.5ex}}

\newcommand{\LowPlus}{\plusdot}  
\newcommand{\UppPlus}{\dotplus}

\newcommand{\PRIMAL}{{\mathbb X}}
\newcommand{\Primal}{X}
\newcommand{\primal}{x}
\newcommand{\DUAL}{{\mathbb Y}}

\newcommand{\dual}{y}
\newcommand{\ExtendedReals}{[-\infty,+\infty]}
\newcommand{\barRR}{\overline{\mathbb R}}
\newcommand{\cardinal}[1]{\module{#1}}

\newcommand{\dom}{\mathrm{dom}}

\newcommand{\convexhull}{\mathrm{co}}
\newcommand{\closedconvexhull}{\overline{\mathrm{co}}}

\newcommand{\fonctionprimal}{f} 
\newcommand{\fonctiondual}{g} 
\newcommand{\fonctionprimalbis}{h} 
\newcommand{\fonctionuncertain}{h} 
\newcommand{\InfimalPostcomposition}[2]{#1\rhd#2} 
 
\newcommand{\lzero}{\ell_0}
\newcommand{\pseudonormlzero}{$l_0$~pseudonorm}
\newcommand{\LevelSet}[2]{#1^{\leq #2}}
\newcommand{\LevelCurve}[2]{#1^{= #2}}

\newcommand{\coupling}{c}
\newcommand{\Capra}{Caprac} 
\newcommand{\couplingCAPRA}{\cent} 
\newcommand{\LFM}[1]{#1^{\star}}
\newcommand{\LFMr}[1]{#1^{\star'}}
\newcommand{\LFMbi}[1]{#1^{\star\star'}}
\newcommand{\minusLFM}[1]{#1^{(-\star)}}
\newcommand{\minusLFMr}[1]{#1^{(-\star)'}}
\newcommand{\minusLFMbi}[1]{#1^{(-\star)(-\star)'}}
\newcommand{\SFM}[2]{#1^{#2}}
\newcommand{\SFMbi}[2]{#1^{#2{#2}'}}
\newcommand{\triplenorm}[1]{||| #1 |||}
\newcommand{\SymmetricGauge}[2]{{#2}_{(#1)}^{\mathrm{sgn}}}
\newcommand{\SymmetricGaugeNorm}[2]{\norm{#2}_{(#1)}^{\mathrm{sgn}}}
\newcommand{\SupportNorm}[2]{\norm{#2}_{(#1)}^{\mathrm{sn}}}
\newcommand{\SupportBall}[2]{{#2}_{(#1)}^{\mathrm{sn}}}

\newcommand{\SPHERE}{S} 
\newcommand{\ESPHERE}{{\mathbb S}} 
\newcommand{\BALL}{B}
\newcommand{\EBALL}{{\mathbb B}}
\newcommand{\normalized}{n}

\title{A Suitable Conjugacy for the $l_0$ Pseudonorm}

\author{Jean-Philippe Chancelier and 
Michel De Lara, \\ CERMICS, \'Ecole des Ponts ParisTech}

\begin{document}

\maketitle

\begin{abstract}
The so-called \pseudonormlzero\ on~$\RR^d$
counts the number of nonzero components of a vector.
It is well-known that the \pseudonormlzero\ is not convex, as 
its Fenchel biconjugate is zero.
In this paper, we introduce a suitable conjugacy, induced by 
a novel coupling, \Capra, 
having the property of being constant along primal rays, 
like the \pseudonormlzero.
The \Capra\ coupling belongs to the class of one-sided linear couplings,
that we introduce. We show that they induce conjugacies 
that share nice properties with the classic Fenchel conjugacy.
For the \Capra\ conjugacy, induced by the coupling \Capra, 
we prove that the \pseudonormlzero\ is equal to its biconjugate:
hence, the \pseudonormlzero\ is \Capra-convex 
in the sense of generalized convexity.
As a corollary, we show that the \pseudonormlzero\ coincides,
on the sphere, with a convex lsc function.
We also provide expressions for conjugates in terms 
of two families of dual norms, the $2$-$k$-symmetric gauge norms
and the $k$-support norms. 
\end{abstract}



\section{Introduction}

The \emph{counting function}, also called \emph{cardinality function}
or \emph{\pseudonormlzero}, 
counts the number of nonzero components of a vector in~$\RR^d$.
It is related to the rank function defined over matrices
\cite{Hiriart-Ururty-Le:2013}.
It is well-known that the \pseudonormlzero\ is 
lower semi continuous but is not convex.
This can be deduced from the computation of its Fenchel biconjugate,
which is zero.

In this paper, we display a suitable conjugacy 
for which we show that the \pseudonormlzero\ is ``convex''
in the sense of generalized convexity (equal to its biconjugate). 
As a corollary, we also show that the \pseudonormlzero\ coincides,
on the sphere, with a convex lsc function.

The paper is organized as follows.
In Sect.~\ref{The_constant_along_primal_rays_coupling},
we provide background on Fenchel-Moreau conjugacies,
then introduce a novel class of \emph{one-sided linear couplings},
which includes the 
\emph{constant along primal rays coupling~$\couplingCAPRA$ (\Capra)}.
We show that one-sided linear couplings induce conjugacies that share nice properties
with the classic Fenchel conjugacy, by giving expressions for 
conjugate and biconjugate functions.
We elucidate the structure 
of \Capra-convex functions.
Then, 
in Sect.~\ref{CAPRAC_conjugates_and_biconjugates_related_to_the_pseudo_norm},
we relate the \Capra\ conjugate and biconjugate 
of the \pseudonormlzero, 
the characteristic functions of its level sets
and the symmetric gauge norms.
In particular, we show that the \pseudonormlzero\ is 
\Capra\ biconjugate (a \Capra-convex function), 
from which we deduce
that it coincides, on the sphere, with a convex lsc function.
The Appendix~\ref{Appendix} gathers 
background on J.~J. Moreau lower and upper additions,
properties of $2$-$k$-symmetric gauge norms, and 
properties of the \pseudonormlzero\ level sets.

\section{The constant along primal rays coupling (\Capra)}
\label{The_constant_along_primal_rays_coupling}

After having recalled background on Fenchel-Moreau conjugacies
in~\S\ref{Background_on_Fenchel-Moreau_conjugacies},
we introduce \emph{one-sided linear couplings}
in~\S\ref{One-sided_linear_couplings}, 
and finally 
the \emph{constant along primal rays coupling~$\couplingCAPRA$ (\Capra)}
in~\S\ref{Constant_along_primal_rays_coupling}.

\subsection{Background on Fenchel-Moreau conjugacies}
\label{Background_on_Fenchel-Moreau_conjugacies}

We review general concepts and notations, then we focus 
on the special case of the Fenchel conjugacy.
We denote \( \barRR=\ExtendedReals \). 
Background on J.~J. Moreau lower and upper additions can be found 
in~\S\ref{Moreau_lower_and_upper_additions}.

\subsubsection*{The general case}

Let be given two sets $\PRIMAL$ (``primal''), $\DUAL$ (``dual''), together 
with a \emph{coupling} function
\begin{equation}
  \coupling : \PRIMAL \times \DUAL \to \barRR 
\eqfinp 
\end{equation}

With any coupling, we associate \emph{conjugacies} 
from \( \barRR^\PRIMAL \) to \( \barRR^\DUAL \) 
and from \( \barRR^\DUAL \) to \( \barRR^\PRIMAL \) 
as follows.

\begin{definition}
  The \emph{$\coupling$-Fenchel-Moreau conjugate} of a 
  function \( \fonctionprimal : \PRIMAL  \to \barRR \), 
  with respect to the coupling~$\coupling$, is
  the function \( \SFM{\fonctionprimal}{\coupling} : \DUAL  \to \barRR \) 
  defined by
  \begin{equation}
    \SFM{\fonctionprimal}{\coupling}\np{\dual} = 
    \sup_{\primal \in \PRIMAL} \Bp{ \coupling\np{\primal,\dual} 
      \LowPlus \bp{ -\fonctionprimal\np{\primal} } } 
    \eqsepv \forall \dual \in \DUAL
    \eqfinp
    \label{eq:Fenchel-Moreau_conjugate}
  \end{equation}
With the coupling $\coupling$, we associate 
the \emph{reverse coupling~$\coupling'$} defined by 
\begin{equation}
  \coupling': \DUAL \times \PRIMAL \to \barRR 
\eqsepv
\coupling'\np{\dual,\primal}= \coupling\np{\primal,\dual} 
\eqsepv
\forall \np{\dual,\primal} \in \DUAL \times \PRIMAL
\eqfinp
    \label{eq:reverse_coupling}
\end{equation}
  The \emph{$\coupling'$-Fenchel-Moreau conjugate} of a 
  function \( \fonctiondual : \DUAL \to \barRR \), 
  with respect to the coupling~$\coupling'$, is
  the function \( \SFM{\fonctiondual}{\coupling'} : \PRIMAL \to \barRR \) 
  defined by
  \begin{equation}
    \SFM{\fonctiondual}{\coupling'}\np{\primal} = 
    \sup_{ \dual \in \DUAL } \Bp{ \coupling\np{\primal,\dual} 
      \LowPlus \bp{ -\fonctiondual\np{\dual} } } 
    \eqsepv \forall \primal \in \PRIMAL 
    \eqfinp
    \label{eq:Fenchel-Moreau_reverse_conjugate}
  \end{equation}
  The \emph{$\coupling$-Fenchel-Moreau biconjugate} of a 
  function \( \fonctionprimal : \PRIMAL  \to \barRR \), 
  with respect to the coupling~$\coupling$, is
  the function \( \SFMbi{\fonctionprimal}{\coupling} : \PRIMAL \to \barRR \) 
  defined by
  \begin{equation}
    \SFMbi{\fonctionprimal}{\coupling}\np{\primal} = 
    \bp{\SFM{\fonctionprimal}{\coupling}}^{\coupling'} \np{\primal} = 
    \sup_{ \dual \in \DUAL } \Bp{ \coupling\np{\primal,\dual} 
      \LowPlus \bp{ -\SFM{\fonctionprimal}{\coupling}\np{\dual} } } 
    \eqsepv \forall \primal \in \PRIMAL 
    \eqfinp
    \label{eq:Fenchel-Moreau_biconjugate}
  \end{equation}
%
%
\end{definition}

  \begin{subequations}
For any coupling~$\coupling$, 
\begin{itemize}
\item 
the biconjugate of a 
  function \( \fonctionprimal : \PRIMAL  \to \barRR \) satisfies
\begin{equation}
  \SFMbi{\fonctionprimal}{\coupling}\np{\primal}
  \leq \fonctionprimal\np{\primal}
  \eqsepv \forall \primal \in \PRIMAL 
  \eqfinv
  \label{eq:galois-cor}
\end{equation}
\item 
for any couple of functions \( \fonctionprimal : \PRIMAL  \to \barRR \) 
and \( \fonctionprimalbis : \PRIMAL \to \barRR \), we have the inequality
\begin{equation}
     \sup_{\dual \in \DUAL} 
    \Bp{ 
\bp{-\SFM{\fonctionprimal}{\coupling}\np{\dual} } 
      \LowPlus 
\bp{-\SFM{\fonctionprimalbis}{-\coupling}\np{\dual} }
}
    \leq 
    \inf_{\primal \in \PRIMAL} 
\Bp{ 
\fonctionprimal\np{\primal} 
      \UppPlus 
\fonctionprimalbis\np{\primal} 
} 
\eqfinv
  \label{eq:dual_problem_inequality}
\end{equation}
where the \emph{$\np{-\coupling}$-Fenchel-Moreau conjugate} is given by
  \begin{equation}
    \SFM{\fonctionprimalbis}{-\coupling}\np{\dual} = 
    \sup_{\primal \in \PRIMAL} \Bp{ \bp{ -\coupling\np{\primal,\dual} }
      \LowPlus \bp{ -\fonctionprimalbis\np{\primal} } } 
    \eqsepv \forall \dual \in \DUAL
    \eqfinv
    \label{eq:minus_Fenchel-Moreau_conjugate}
  \end{equation}
\item 
for any function \( \fonctionprimal : \PRIMAL  \to \barRR \) 
and subset \( \Primal \subset \PRIMAL \), we have the inequality
\begin{equation}
      \sup_{\dual \in \DUAL} 
    \Bp{ 
\bp{-\SFM{\fonctionprimal}{\coupling}\np{\dual} } 
      \LowPlus 
\bp{-\SFM{\delta_{\Primal}}{-\coupling}\np{\dual} }
}
    \leq 
    \inf_{\primal \in \PRIMAL} 
\Bp{ 
\fonctionprimal\np{\primal} 
      \UppPlus 
\delta_{\Primal} \np{\primal} } 
= \inf_{\primal \in \Primal} \fonctionprimal\np{\primal} 
\eqfinp
  \label{eq:dual_problem_inequality_constraints}
\end{equation}
\end{itemize}
  \end{subequations}

\subsubsection*{The Fenchel conjugacy}

When the sets $\PRIMAL$ and $\DUAL$ are vector spaces 
equipped with a bilinear form \( \proscal{}{} \),
the corresponding conjugacy is the classical 
\emph{Fenchel conjugacy}.
For any functions \( \fonctionprimal : \PRIMAL  \to \barRR \)
and \( \fonctiondual : \DUAL \to \barRR \), we denote
\begin{subequations}
\begin{align}
    \LFM{\fonctionprimal}\np{\dual} 
&= 
    \sup_{\primal \in \PRIMAL} \Bp{ \proscal{\primal}{\dual} 
      \LowPlus \bp{ -\fonctionprimal\np{\primal} } } 
    \eqsepv \forall \dual \in \DUAL
    \eqfinv
    \label{eq:Fenchel_conjugate}
  \\
    \LFMr{\fonctiondual}\np{\primal} 
&= 
    \sup_{ \dual \in \DUAL } \Bp{ \proscal{\primal}{\dual} 
      \LowPlus \bp{ -\fonctiondual\np{\dual} } } 
    \eqsepv \forall \primal \in \PRIMAL 
    \label{eq:Fenchel_conjugate_reverse}
\\
    \LFMbi{\fonctionprimal}\np{\primal} 
&= 
    \sup_{\dual \in \DUAL} \Bp{ \proscal{\primal}{\dual} 
      \LowPlus \bp{ -\LFM{\fonctionprimal}\np{\dual} } } 
    \eqsepv \forall \primal \in \PRIMAL
    \eqfinp
    \label{eq:Fenchel_biconjugate}
\end{align}
\end{subequations}
Due to the presence of the coupling \( \np{-\coupling} \) 
in the Inequality~\eqref{eq:dual_problem_inequality},
we also introduce\footnote{%
In convex analysis, one does not use the notations below, but rather uses 
\( \fonctionprimal^{\lor}\np{\primal} =\fonctionprimal\np{-\primal} \), 
for all \( \primal \in \PRIMAL \), 
and
\( \fonctiondual^{\lor}\np{\dual}=\fonctiondual\np{-\dual} \), 
for all \( \dual \in \DUAL \).
The connection between both notations is given by 
\(  \minusLFM{\fonctionprimal} =
\LFM{ \bp{\fonctionprimal^{\lor}} } = 
\bp{\LFM{\fonctionprimal}}^{\lor} \).
}
\begin{subequations}
\begin{align}
    \minusLFM{\fonctionprimal}\np{\dual} 
&= 
    \sup_{\primal \in \PRIMAL} \Bp{ -\proscal{\primal}{\dual} 
      \LowPlus \bp{ -\fonctionprimal\np{\primal} } } 
= \LFM{\fonctionprimal}\np{-\dual} 
    \eqsepv \forall \dual \in \DUAL
    \eqfinv
    \label{eq:minusFenchel_conjugate}
  \\
    \minusLFMr{\fonctiondual}\np{\primal} 
&= 
    \sup_{ \dual \in \DUAL } \Bp{ -\proscal{\primal}{\dual} 
      \LowPlus \bp{ -\fonctiondual\np{\dual} } } 
= \LFMr{\fonctiondual}\np{-\primal} 
    \eqsepv \forall \primal \in \PRIMAL 
    \label{eq:minusFenchel_conjugate_reverse}
\\
    \minusLFMbi{\fonctionprimal}\np{\primal} 
&= 
    \sup_{\dual \in \DUAL} \Bp{ -\proscal{\primal}{\dual} 
      \LowPlus \bp{ -\minusLFM{\fonctionprimal}\np{\dual} } } 
=\LFMbi{\fonctionprimal}\np{\primal}
    \eqsepv \forall \primal \in \PRIMAL
    \eqfinp
    \label{eq:minusFenchel_biconjugate}
\end{align}
\end{subequations}



When the two vector spaces $\PRIMAL$ and $\DUAL$ are \emph{paired}
in the sense of convex analysis\footnote{That is,
$\PRIMAL$ and $\DUAL$ are equipped with 
a bilinear form \( \proscal{}{} \), 
and locally convex topologies
that are compatible in the sense
that the continuous linear forms on~$\PRIMAL$
are the functions 
\( \primal \in \PRIMAL \mapsto \proscal{\primal}{\dual} \),
for all \( \dual \in \DUAL \),
and 
that the continuous linear forms on~$\DUAL$
are the functions 
\( \dual \in \DUAL \mapsto \proscal{\primal}{\dual} \),
for all \( \primal \in \PRIMAL \)},
Fenchel conjugates are convex 
\emph{lower semi continuous (lsc)} functions,
and their opposites are concave
\emph{upper semi continuous (usc)} functions.

\subsection{One-sided linear couplings}
\label{One-sided_linear_couplings}

Let $\UNCERTAIN$ and $\PRIMAL$ be any two sets and
\( \theta: \UNCERTAIN \to \PRIMAL \) be a mapping.
We recall the definition \cite[p.~214]{Bauschke-Combettes:2017}
of the \emph{infimal postcomposition} 
\( \bp{\InfimalPostcomposition{\theta}{\fonctionuncertain}}:
\PRIMAL \to \barRR \) of 
a function \( \fonctionuncertain : \UNCERTAIN \to \barRR \): 
  \begin{equation}
\bp{\InfimalPostcomposition{\theta}{\fonctionuncertain}}\np{\primal}
=
\inf\defset{\fonctionuncertain\np{\uncertain}}{%
\uncertain\in\UNCERTAIN \eqsepv \theta\np{\uncertain}=\primal}
\eqsepv \forall \primal \in \PRIMAL 
  \eqfinv
\label{eq:InfimalPostcomposition}
  \end{equation}
with the convention that \( \inf \emptyset = +\infty \)
(and with the consequence that 
\( \theta: \UNCERTAIN \to \PRIMAL \) need not be defined
on all~$\UNCERTAIN$, but only on~\( \dom\fonctionuncertain \)).
The infimal postcomposition has the following 
\emph{invariance property}
\begin{equation}
  \fonctionuncertain = \fonctionprimal \circ \theta
\mtext{ where } \fonctionprimal: \PRIMAL  \to \barRR 
\Rightarrow
\InfimalPostcomposition{\theta}{\fonctionuncertain}
= \fonctionprimal \UppPlus \delta_{ \theta\np{\UNCERTAIN} }
\eqfinv 
\label{eq:InfimalPostcomposition_invariance}
\end{equation}
where $\delta_{Z}$ denotes the \emph{characteristic function} of a set~$Z$:
\begin{equation}
  \delta_{Z}\np{z} =  
   \begin{cases}
      0 &\text{ if } z \in Z \eqfinv \\ 
      +\infty &\text{ if } z \not\in Z \eqfinp
    \end{cases} 
\label{eq:characteristic_function}
\end{equation}

\begin{definition}
Let $\PRIMAL$ and $\DUAL$ be two vector spaces equipped with a bilinear
form \( \proscal{}{} \).
Let $\UNCERTAIN$ be a set and
\( \theta: \UNCERTAIN \to \PRIMAL \) a mapping.
  We define the \emph{one-sided linear coupling} $\coupling_{\theta}$
between $\UNCERTAIN$ and $\DUAL$ by
\begin{equation}
\coupling_{\theta}: \UNCERTAIN \times \DUAL \to \barRR 
\eqsepv 
\coupling_{\theta}\np{\uncertain, \dual} = 
\proscal{\theta\np{\uncertain}}{\dual} 
\eqsepv \forall \uncertain \in \UNCERTAIN
\eqsepv \forall \dual \in \DUAL
\eqfinp 
\label{eq:one-sided_linear_coupling}
\end{equation}
\end{definition}

Here are expressions for the 
conjugates and biconjugates of a function. 
\begin{proposition}
  For any function \( \fonctiondual : \DUAL \to \barRR \), 
  the $\coupling_{\theta}'$-Fenchel-Moreau conjugate is given by 
  \begin{equation}
\SFM{\fonctiondual}{\coupling_{\theta}'}=
 \LFM{ \fonctiondual } \circ \theta
\eqfinp
\label{eq:one-sided_linear_c'-Fenchel-Moreau_conjugate}
\end{equation}
  For any function \( \fonctionuncertain : \UNCERTAIN \to \barRR \), 
  the $\coupling_{\theta}$-Fenchel-Moreau conjugate is given by 
   \begin{equation}
\SFM{\fonctionuncertain}{\coupling_{\theta}}=
 \LFM{ \bp{\InfimalPostcomposition{\theta}{\fonctionuncertain}} }
\eqfinv 
\label{eq:one-sided_linear_Fenchel-Moreau_conjugate}
  \end{equation}
and the $\coupling_{\theta}$-Fenchel-Moreau biconjugate 
is given by
\begin{equation}
  \SFMbi{\fonctionuncertain}{\coupling_{\theta}}
= 
\SFM{ \bp{ \SFM{\fonctionuncertain}{\coupling_{\theta}} } }{\star} 
\circ \theta
=
 \LFMbi{ \bp{\InfimalPostcomposition{\theta}{\fonctionuncertain}} }
\circ \theta
\eqfinp
\label{eq:one-sided_linear_Fenchel-Moreau_biconjugate}
\end{equation}
For any subset \( \Uncertain \subset \UNCERTAIN \), 
the $\np{-\coupling_{\theta}}$-Fenchel-Moreau conjugate 
of the characteristic function of~$\Uncertain$ 
is given by 
\begin{equation}
\SFM{ \delta_{\Uncertain} }{-\coupling_{\theta}}
= \sigma_{ -\theta\np{\Uncertain} } 
\eqfinp
\label{eq:one-sided_linear_Fenchel-Moreau_characteristic}
\end{equation}
  \label{pr:one-sided_linear_Fenchel-Moreau_conjugate}
\end{proposition}
We recall that, in convex analysis,
\( \sigma_{\Primal} : \DUAL \to \barRR\) denotes the 
\emph{support function of a subset~$\Primal\subset\PRIMAL$}:
\begin{equation}
  \sigma_{\Primal}\np{\dual} = 
\sup_{\primal\in\Primal} \proscal{\primal}{\dual}
\eqsepv \forall \dual \in \DUAL
\eqfinp
\label{eq:support_function}
\end{equation}

\begin{proof}
We prove~\eqref{eq:one-sided_linear_c'-Fenchel-Moreau_conjugate}.
Letting \( \uncertain \in \UNCERTAIN  \), we have that 
  \begin{align*}
\SFM{ \bp{ \fonctiondual } }{\coupling_{\theta}'}\np{\uncertain}
    &=
       \sup_{\dual \in \DUAL} 
\Bp{ \proscal{\theta\np{\uncertain}}{\dual}
      \LowPlus \bp{ -
 \fonctiondual\np{\dual} } }
\tag{by the conjugate formula~\eqref{eq:Fenchel-Moreau_conjugate} 
and the coupling~\eqref{eq:one-sided_linear_coupling}}
\\
    &=
 \LFM{ \fonctiondual }\bp{\theta\np{\uncertain} } 
\tag{by the expression~\eqref{eq:Fenchel_conjugate} of the Fenchel conjugate}
\eqfinp
  \end{align*}

We prove~\eqref{eq:one-sided_linear_Fenchel-Moreau_conjugate}.
Letting \( \dual \in \DUAL \), we have that 
  \begin{align*}
    \SFM{\fonctionuncertain}{\coupling_{\theta}}\np{\dual} 
    &=
       \sup_{\uncertain \in \PRIMAL }
\Bp{ \proscal{\theta\np{\uncertain}}{\dual}
      \LowPlus \bp{ -\fonctionuncertain\np{\uncertain} } }
\tag{by the conjugate formula~\eqref{eq:Fenchel-Moreau_conjugate} 
and the coupling~\eqref{eq:one-sided_linear_coupling}}
    \\
    &=
       \sup_{\primal \in \PRIMAL }
       \sup_{\uncertain \in \PRIMAL, \theta\np{\uncertain}=\primal}
\Bp{ \proscal{\theta\np{\uncertain}}{\dual}
      \LowPlus \bp{ -\fonctionuncertain\np{\uncertain} } }
    \\
    &=
       \sup_{\primal \in \PRIMAL }
\Bp{ \proscal{\primal}{\dual}
      \LowPlus        \sup_{\uncertain \in \PRIMAL, \theta\np{\uncertain}=\primal}
\bp{ -\fonctionuncertain\np{\uncertain} } }
\tag{by~\eqref{eq:lower_addition_sup}}
    \\
    &=
       \sup_{\primal \in \PRIMAL }
\Bp{ \proscal{\primal}{\dual}
      \LowPlus  \bp{ - \inf_{\uncertain \in \PRIMAL, \theta\np{\uncertain}=\primal}
\fonctionuncertain\np{\uncertain} } }
    \\
    &=
       \sup_{\primal \in \PRIMAL }
\bgp{ \proscal{\primal}{\dual}
      \LowPlus  \Bp{ - 
\bp{\InfimalPostcomposition{\theta}{\fonctionuncertain}}\np{\primal} } }
\tag{by the infimal postcomposition 
expression~\eqref{eq:InfimalPostcomposition}}
    \\
    &=
 \LFM{ \bp{\InfimalPostcomposition{\theta}{\fonctionuncertain}} }\np{\dual}
\tag{by the expression~\eqref{eq:Fenchel_conjugate} of the Fenchel conjugate}
  \end{align*}

We prove~\eqref{eq:one-sided_linear_Fenchel-Moreau_biconjugate}.
Letting \( \primal \in \PRIMAL  \), we have that 
  \begin{align*}
    \SFMbi{\fonctionuncertain}{\coupling_{\theta}}\np{\primal} 
    &=
\SFM{ \bp{ \SFM{\fonctionuncertain}{\coupling_{\theta}} } }{\coupling_{\theta}'}\np{\primal}
\tag{by the definition~\eqref{eq:Fenchel-Moreau_biconjugate}
of the biconjugate}
\\
    &=
\SFM{ \bp{ \LFM{ \bp{\InfimalPostcomposition{\theta}{\fonctionuncertain}} } } }%
{\coupling_{\theta}'}\np{\primal}
\tag{by~\eqref{eq:one-sided_linear_Fenchel-Moreau_conjugate}}
\\
    &=
       \sup_{\dual \in \DUAL} 
\Bp{ \proscal{\theta\np{\primal}}{\dual}
      \LowPlus \bp{ -
\LFM{ \bp{\InfimalPostcomposition{\theta}{\fonctionuncertain}} }\np{\dual} } }
\tag{by the conjugate formula~\eqref{eq:Fenchel-Moreau_conjugate} 
and the coupling~\eqref{eq:one-sided_linear_coupling}}
\\
    &=
 \LFMbi{ \bp{\InfimalPostcomposition{\theta}{\fonctionuncertain}} }
\bp{\theta\np{\primal} } 
\tag{by the expression~\eqref{eq:Fenchel_conjugate} of the Fenchel conjugate}
  \end{align*}

We prove~\eqref{eq:one-sided_linear_Fenchel-Moreau_characteristic}:
\begin{align*}
  \SFM{ \delta_{\Uncertain} }{-\coupling_{\theta}}
&= 
\SFM{ \delta_{\Uncertain} }{\coupling_{\np{-\theta}}} 
\tag{ because \( -\coupling_{\theta}=\coupling_{\np{-\theta}} \)
by~\eqref{eq:one-sided_linear_coupling} }
\\
&= 
 \LFM{ \bp{\InfimalPostcomposition{\np{-\theta}}{ \delta_{\Uncertain} }} }
\tag{ by~\eqref{eq:one-sided_linear_Fenchel-Moreau_conjugate} }
\\
&= 
 \LFM{ \delta_{ -\theta\np{\Uncertain} } }
\tag{ because \( \InfimalPostcomposition{\theta}{\delta_{\Uncertain}} =
\delta_{ \theta\np{\Uncertain} } \) by~\eqref{eq:InfimalPostcomposition} }
\\
&= 
 \sigma_{ -\theta\np{\Uncertain} } 
\tag{ as is well-known in convex analysis }
\eqfinp
\end{align*}
This ends the proof.
\end{proof}

\subsection{Constant along primal rays coupling (\Capra)}
\label{Constant_along_primal_rays_coupling}

Now, we introduce a novel coupling, which is a special case of
one-sided linear couplings.

\begin{definition}
Let $\PRIMAL$ and $\DUAL$ be two vector spaces 
equipped with a bilinear form \( \proscal{}{} \), 
and suppose that $\PRIMAL$ is equipped with a norm~$\triplenorm{\cdot}$.
  We define the \emph{\Capra\ coupling}~$\couplingCAPRA$ 
between $\PRIMAL$ and $\DUAL$ by
\begin{equation}
\forall \dual \in \DUAL \eqsepv 
\begin{cases}
  \couplingCAPRA\np{\primal, \dual} 
&= \displaystyle
\frac{ \proscal{\primal}{\dual} }{ \triplenorm{\primal} }
\eqsepv \forall \primal \in \PRIMAL\backslash\{0\}
\\[4mm]
\couplingCAPRA\np{0, \dual} &= 0.
\end{cases}
\label{eq:coupling_CAPRAC_original}
\end{equation}
\end{definition}
We stress the point that, in~\eqref{eq:coupling_CAPRAC_original},
the bilinear form term \( \proscal{\primal}{\dual} \)
and the norm term \( \triplenorm{\primal} \) need not be related.
Indeed, the bilinear form \( \proscal{}{} \) is not necessarily a scalar
product and the norm~$\triplenorm{\cdot}$ is not necessarily induced by this latter.

The \Capra\ coupling has the property of being 
\emph{constant along primal rays}, hence the acronym~\Capra.
We introduce the unit sphere~$\SPHERE_{\triplenorm{\cdot}}$
of the normed space~$\bp{\PRIMAL,\triplenorm{\cdot}}$, 
and the primal \emph{normalization mapping}~$\normalized$
\begin{equation}
\SPHERE_{\triplenorm{\cdot}}= 
\defset{\primal \in \PRIMAL}{\triplenorm{\primal} = 1} 
\mtext{ and }
\normalized : \PRIMAL \to \SPHERE_{\triplenorm{\cdot}} \cup \{0\} 
\eqsepv
\normalized\np{\primal}=
\begin{cases}
\frac{ \primal }{ \triplenorm{\primal} }
& \mtext{ if } \primal \neq 0 \eqfinv 
\\
0 
& \mtext{ if } \primal = 0 \eqfinp
\end{cases}  
\label{eq:normalization_mapping}
\end{equation}
We immedialy obtain that, 
for all subset \( D \subset \RR^d \) that contains zero (\( 0 \in D \)):
\begin{equation}
\normalized^{-1}\np{D} 
= \normalized^{-1}\bp{ \np{\{0\} \cup \SPHERE_{\triplenorm{\cdot}}} \cap D } 
= \{0\} \cup \normalized^{-1}\np{\SPHERE_{\triplenorm{\cdot}} \cap D } 
\eqfinp
\label{eq:pre-image_normalization_mapping}
\end{equation}

With these notations, the  \Capra\ 
coupling~\eqref{eq:coupling_CAPRAC_original} is a special case
of one-sided linear coupling $\coupling_{\normalized}$, 
as in~\eqref{eq:one-sided_linear_coupling} with \(\theta=\normalized\),
the Fenchel coupling after primal normalization:
\begin{equation}
\couplingCAPRA\np{\primal, \dual} = 
\coupling_{\normalized}\np{\primal, \dual} = 
\proscal{\normalized\np{\primal}}{\dual} 
\eqsepv \forall \primal \in \PRIMAL
\eqsepv \forall \dual \in \DUAL 
\eqfinp
\label{eq:coupling_CAPRAC}
\end{equation}

Here are expressions for the  \Capra-conjugates and biconjugates
of a function. 
The following Proposition simply is
Proposition~\ref{pr:one-sided_linear_Fenchel-Moreau_conjugate}
in the case where the mapping \(\theta \) is the
normalization mapping~\(\normalized\) in~\eqref{eq:normalization_mapping}. 

\begin{proposition}
Let $\PRIMAL$ and $\DUAL$ be two vector spaces 
equipped with a bilinear form \( \proscal{}{} \), 
and suppose that $\PRIMAL$ is equipped with a norm~$\triplenorm{\cdot}$.

  For any function \( \fonctiondual : \DUAL \to \barRR \), 
  the $\couplingCAPRA'$-Fenchel-Moreau conjugate is given by 
  \begin{equation}
\SFM{\fonctiondual}{\couplingCAPRA'}=
 \LFM{ \fonctiondual } \circ \normalized
\eqfinp 
\label{eq:CAPRA'_Fenchel-Moreau_conjugate}
\end{equation}
  For any function \( \fonctionprimal : \PRIMAL \to \barRR \), 
  the $\couplingCAPRA$-Fenchel-Moreau conjugate is given by 
   \begin{equation}
\SFM{\fonctionprimal}{\couplingCAPRA}=
 \LFM{ \bp{\InfimalPostcomposition{\normalized}{\fonctionprimal}} }
\eqfinv 
\label{eq:CAPRA_Fenchel-Moreau_conjugate}
  \end{equation}
where the infimal postcomposition~\eqref{eq:InfimalPostcomposition}
has the expression
  \begin{equation}
\bp{\InfimalPostcomposition{\normalized}{\fonctionprimal}}\np{\primal}
=
\inf\defset{\fonctionprimal\np{\primal'}}{
\normalized\np{\primal'}=\primal}
=
\begin{cases}
  \inf_{\lambda > 0} \fonctionprimal\np{\lambda\primal}
& \text{if } \primal \in \SPHERE_{\triplenorm{\cdot}}  \cup \{0\} 
\\
+\infty  
& \text{if } \primal \not\in \SPHERE_{\triplenorm{\cdot}}  \cup \{0\} 
\end{cases}
\label{eq:CAPRA_InfimalPostcomposition}
  \end{equation}
and the $\couplingCAPRA$-Fenchel-Moreau biconjugate 
is given by
\begin{equation}
  \SFMbi{\fonctionprimal}{\couplingCAPRA}
= 
\SFM{ \bp{ \SFM{\fonctionprimal}{\couplingCAPRA} } }{\star} 
\circ \normalized
=
 \LFMbi{ \bp{\InfimalPostcomposition{\normalized}{\fonctionprimal}} }
\circ \normalized
\eqfinp
\label{eq:CAPRA_Fenchel-Moreau_biconjugate}
\end{equation}
  \label{pr:CAPRA_Fenchel-Moreau_conjugate}
\end{proposition}

We recall that so-called  \Capra\ $\couplingCAPRA$-convex functions 
are all functions \( \fonctionprimal : \PRIMAL \to \barRR \) of the form 
$\SFM{ \bp{ \fonctiondual } }{\couplingCAPRA'}$, 
for any \( \fonctiondual \in \barRR^\DUAL \), 
or, equivalently,
all functions of the form $\SFMbi{\fonctionprimal}{\couplingCAPRA}$, 
for any \( \fonctionprimal \in \barRR^\PRIMAL \), 
or, equivalently,
all functions that are equal to their $\couplingCAPRA$-biconjugate
(\( \SFMbi{\fonctionprimal}{\couplingCAPRA}=\fonctionprimal \))
\cite{Singer:1997,Rubinov:2000,Martinez-Legaz:2005}.
From the expressions~\eqref{eq:CAPRA'_Fenchel-Moreau_conjugate},
\eqref{eq:CAPRA_Fenchel-Moreau_conjugate}
and \eqref{eq:CAPRA_Fenchel-Moreau_biconjugate}, we easily deduce the 
following result.

\begin{corollary}
When $\PRIMAL$ and $\DUAL$ are two paired vector spaces,
and $\PRIMAL$ is equipped with a norm~$\triplenorm{\cdot}$,
the $\couplingCAPRA$-Fenchel-Moreau conjugate 
\( \SFM{\fonctionprimal}{\couplingCAPRA} \)
  is a convex lower semi continuous (lsc) function on~$\DUAL$.
In addition, using~\eqref{eq:CAPRA'_Fenchel-Moreau_conjugate}, 
a function is $\couplingCAPRA$-convex 
if and only if it is the composition of 
a convex lower semi continuous function on~$\PRIMAL$
with the normalization mapping~\eqref{eq:normalization_mapping}.
\end{corollary}

\section{ \Capra\ conjugates and biconjugates related to the \pseudonormlzero}
\label{CAPRAC_conjugates_and_biconjugates_related_to_the_pseudo_norm}

In this Section, we work on the Euclidian space~$\RR^d$
(with $d \in \NN^*$), equipped with the scalar product 
\( \proscal{\cdot}{\cdot} \)
and with the Euclidian norm
\( \norm{\cdot} = \sqrt{ \proscal{\cdot}{\cdot} } \).
In particular, we consider the \emph{Euclidian unit sphere}
\begin{equation}
  \SPHERE= \defset{\primal \in \PRIMAL}{\norm{\primal} = 1} 
\eqfinv
\label{eq:Euclidian_SPHERE}
\end{equation}
and the \emph{(Euclidian) coupling \Capra}~$\couplingCAPRA$ 
between $\RR^d$ and $\RR^d$ by
\begin{equation}
\forall \dual \in \RR^d \eqsepv 
\begin{cases}
  \couplingCAPRA\np{\primal, \dual} 
&= \displaystyle
\frac{ \proscal{\primal}{\dual} }{ \norm{\primal} }
\eqsepv \forall \primal \in \RR^d\backslash\{0\} \eqfinv
\\[4mm]
\couplingCAPRA\np{0, \dual} &= 0.
\end{cases}
\label{eq:Euclidian_coupling_CAPRAC}
\end{equation}

The so-called \emph{\pseudonormlzero} is the function
\( \lzero : \RR^d \to \ba{0,1,\ldots,d} \)
defined, for any \( \primal \in \RR^d \), by
\begin{equation}
\lzero\np{\primal} =
\module{\primal}_{0} = \textrm{number of nonzero components of } \primal
\eqfinp
\label{eq:pseudo_norm_l0}  
\end{equation}
The \pseudonormlzero\ displays the invariance property
\begin{equation}
\lzero \circ \normalized = \lzero   
\label{eq:pseudonormlzero_invariance_normalized}
\end{equation}
with respect to the normalization mapping~\eqref{eq:normalization_mapping}.
This property will be instrumental to 
show that the \pseudonormlzero\ 
is a  \Capra\ $\couplingCAPRA$-convex function.
For this purpose, we will start by introducing two dual norms.
\medskip


For any \( \primal \in \RR^d \) and \( K \subset \ba{1,\ldots,d} \), 
we denote by
\( \primal_K \in \RR^d \) the vector which coincides with~\( \primal \),
except for the components outside of~$K$ that vanish:
\( \primal_K \) is the orthogonal projection of~\( \primal \) onto
the subspace \( \RR^K \times \{0\}^{-K} \subset \RR^d \).
Here, following notation from Game Theory, 
we have denoted by $-K$ the complementary subset 
of~$K$ in \( \ba{1,\ldots,d} \): \( K \cup (-K) = \ba{1,\ldots,d} \)
and \( K \cap (-K) = \emptyset \). 
In what follows, $\cardinal{K}$ denotes the cardinal of the set~$K$
and the notation \( \sup_{\cardinal{K} \leq k} \) 
is a shorthand for \( \sup_{ { K \subset \na{1,\ldots,d}, \cardinal{K} \leq k}} \)
(the same holds for \( \sup_{\cardinal{K} = k} \)).

\begin{definition}
  Let \( \primal \in \RR^d \).
  For \( k \in \ba{1,\ldots,d} \), 
we denote by \( \SymmetricGaugeNorm{k}{\primal} \)
  the maximum of \( \norm{\primal_K} \) over all subsets \( K \subset
  \ba{1,\ldots,d} \) with cardinal (less than)~$k$:
  \begin{equation}
\SymmetricGaugeNorm{k}{\primal}
=
\sup_{\cardinal{K} \leq k}\norm{\primal_K} 
=
\sup_{\cardinal{K} = k}\norm{\primal_K} 
    \eqfinp
    \label{eq:symmetric_gauge_norm}
  \end{equation}
  Thus defined, \( \SymmetricGaugeNorm{k}{\cdot} \) is a norm,
the \emph{$2$-$k$-symmetric gauge norm},
or \emph{Ky Fan vector norm}.
Its dual norm (see Definition~\ref{de:dual_norm}) 
is called \emph{$k$-support norm}
\cite{Argyriou-Foygel-Srebro:2012}, denoted by
\( \SupportNorm{k}{\cdot} \):
\begin{equation}
  \SupportNorm{k}{\cdot} = \bp{ \SymmetricGaugeNorm{k}{\cdot} }_{\star}
\eqfinp
    \label{eq:support_norm}
\end{equation}
    \label{de:symmetric_gauge_norm}
\end{definition}
The property that \( \sup_{\cardinal{K} \leq k}\norm{\primal_K} 
= \sup_{\cardinal{K} = k}\norm{\primal_K} \) 
in~\eqref{eq:symmetric_gauge_norm} 
comes from the easy observation that 
\( K\subset K' \Rightarrow \norm{\primal_K} \leq \norm{\primal_{K'}} \).
\bigskip


The \pseudonormlzero\ is used in exact sparse optimization problems of the form
\( \inf_{ \module{\primal}_{0} \leq k } \fonctionprimal\np{\primal} \).
This is why we introduce the \emph{level sets}
\begin{subequations}
  \begin{align}
  \LevelSet{\lzero}{k} 
&= 
\defset{ \primal \in \RR^d }{ \lzero\np{\primal} \leq k }
\eqsepv \forall k \in \ba{0,1,\ldots,d} 
\eqfinv
\label{eq:level_set_pseudonormlzero}
\intertext{ and the \emph{level curves}}
  \LevelCurve{\lzero}{k} 
&= 
\defset{ \primal \in \RR^d }{ \lzero\np{\primal} = k }
\eqsepv \forall k \in \ba{0,1,\ldots,d} 
\eqfinp
\label{eq:level_curve_pseudonormlzero}
  \end{align}
\end{subequations}

  The \pseudonormlzero\ in~\eqref{eq:pseudo_norm_l0}, 
the characteristic functions \( \delta_{ \LevelSet{\lzero}{k} } \) 
of its level sets
and the symmetric gauge norms in~\eqref{eq:symmetric_gauge_norm}
are related by the following conjugate formulas.
The proof relies on results gathered in the Appendix~\ref{Appendix}.

\begin{theorem}
Let~$\couplingCAPRA$ be the Euclidian coupling 
\Capra~\eqref{eq:Euclidian_coupling_CAPRAC}.
Let \( k \in \ba{0,1,\ldots,d} \). We have that:
  \begin{subequations}
    \begin{align}
      \SFM{ \delta_{ \LevelSet{\lzero}{k} } }{-\couplingCAPRA}
=       \SFM{ \delta_{ \LevelSet{\lzero}{k} } }{\couplingCAPRA}
      &=
        \SymmetricGaugeNorm{k}{\cdot} 
\eqfinv
        \label{eq:conjugate_delta_l0norm}
      \\ 
      \SFMbi{ \delta_{ \LevelSet{\lzero}{k} } }{\couplingCAPRA} 
      &=
        \delta_{ \LevelSet{\lzero}{k} } 
\eqfinv
        \label{eq:biconjugate_delta_l0norm}
      \\
      \SFM{ \lzero }{\couplingCAPRA} 
      &=
\sup_{l=0,1,\ldots,d} \Bc{ \SymmetricGaugeNorm{l}{\cdot} -l }  
\eqfinv
        \label{eq:conjugate_l0norm}
      \\
      \SFMbi{ \lzero }{\couplingCAPRA} 
      &=\lzero 
        \label{eq:biconjugate_l0norm}
        \eqfinv
    \end{align}
  \end{subequations}
with the convention, in~\eqref{eq:conjugate_delta_l0norm}
and in~\eqref{eq:conjugate_l0norm}, 
that \( \SymmetricGaugeNorm{0}{\cdot} = 0 \).
\end{theorem}

\begin{proof} 
We will use the framework and results of
Sect.~\ref{The_constant_along_primal_rays_coupling}
with \( \PRIMAL=\DUAL=\RR^d \),
equipped with the scalar product \( \proscal{\cdot}{\cdot} \)
and with the Euclidian norm
\( \norm{\cdot} = \sqrt{ \proscal{\cdot}{\cdot} } \).
\bigskip

\noindent $\bullet$ 
We prove~\eqref{eq:conjugate_delta_l0norm}:
\begin{align*}
\SFM{ \delta_{ \LevelSet{\lzero}{k} } }{-\couplingCAPRA}
&= 
\sigma_{ -\normalized\np{ \LevelSet{\lzero}{k} } }
\tag{by~\eqref{eq:one-sided_linear_Fenchel-Moreau_characteristic}}
\\
&= 
\sigma_{ \normalized\np{ \LevelSet{\lzero}{k} } }
\tag{ by symmetry of the set \( \LevelSet{\lzero}{k} \) 
and of the mapping~\( \normalized \) }
\\
&= 
\SFM{ \delta_{ \LevelSet{\lzero}{k} } }{\couplingCAPRA} 
\tag{by~\eqref{eq:one-sided_linear_Fenchel-Moreau_characteristic}}
\\
&= 
\sigma_{ \normalized\np{  \LevelSet{\lzero}{k} } }
\tag{by~\eqref{eq:one-sided_linear_Fenchel-Moreau_characteristic}}
\\
&= 
\sigma_{ \bp{ \SPHERE  \cap \LevelSet{\lzero}{k}} \cup \{0\} } 
\tag{by the expression~\eqref{eq:normalization_mapping}
of the normalization mapping~$\normalized$}
\\
&= 
\sup \ba{ \sigma_{ \LevelSet{\lzero}{k} \cap \SPHERE }, 0 }
\tag{ as is well-known in convex analysis }
\\
&= 
\sup \ba{ \sigma_{ \bigcup_{ {\cardinal{K} \leq k}} \SPHERE_{K} } , 0 }
\tag{ as \( \LevelSet{\lzero}{k} \cap \SPHERE =
 \bigcup_{ {\cardinal{K} \leq k}} \SPHERE_{K} \)
by~\eqref{eq:level_set_l0_inter_sphere_a} }
\\
&= 
\sup \ba{ 
\sup_{ \cardinal{K} \leq k} \sigma_{ \SPHERE_{K} }, 0 }
\tag{ as is well-known in convex analysis }
\\
&= 
\sup \ba{ \SymmetricGaugeNorm{k}{\cdot}, 0 }
\tag{ as \( \sup_{ \cardinal{K} \leq k} \sigma_{ \SPHERE_{K} }
= \SymmetricGaugeNorm{k}{\cdot} \) by~\eqref{eq:k-norm-from-support}}
\\
&= 
\SymmetricGaugeNorm{k}{\cdot} 
\eqfinp
\end{align*}
\medskip

\noindent $\bullet$ We prove~\eqref{eq:biconjugate_delta_l0norm}:
\begin{align*}
\SFMbi{ \delta_{ \LevelSet{\lzero}{k} } }{\couplingCAPRA}
&=
\LFM{ \bp{ \SFM{ \delta_{ \LevelSet{\lzero}{k} } }{\couplingCAPRA} }}
\circ \normalized
  \tag{ by the formula~\eqref{eq:CAPRA_Fenchel-Moreau_biconjugate} 
for the biconjugate }
\\
&=
\LFM{ \bp{ \SymmetricGaugeNorm{k}{\cdot} } }
\circ \normalized
  \tag{ by~\eqref{eq:conjugate_delta_l0norm} }
\\
&=
\LFM{ \bp{ \sigma_{ \SupportBall{k}{\BALL} } } }
\circ \normalized
  \tag{ by~\eqref{eq:norm_dual_norm}, that expresses a norm as 
a support function } 
\\
&=
\delta_{ \SupportBall{k}{\BALL} } \circ \normalized
\tag{ as \( \LFM{ \bp{ \sigma_{ \SupportBall{k}{\BALL} } } }=
\delta_{ \SupportBall{k}{\BALL} } \) since $\SupportBall{k}{\BALL}$
is closed convex }
\\
&=
\delta_{ \normalized^{-1}\np{\SupportBall{k}{\BALL} } }
\tag{ by the definition~\eqref{eq:characteristic_function} of a characteristic
  function }
\\
&=
\delta_{ \{0\} \cup \normalized^{-1}\np{\SPHERE \cap \SupportBall{k}{\BALL} } }
\tag{by~\eqref{eq:pre-image_normalization_mapping}
since $0\in\SupportBall{k}{\BALL}$ }
\\
&=
\delta_{ \{0\} \cup \normalized^{-1}
\np{ \SPHERE \cap \LevelSet{\lzero}{k} } }
  \tag{ as \( \SPHERE \cap \SupportBall{k}{\BALL}
=  \SPHERE \cap \LevelSet{\lzero}{k} \) 
by~\eqref{eq:level_set_l0_inter_sphere_b} }
\\
&=
\delta_{ \normalized^{-1}\np{\LevelSet{\lzero}{k}}  } 
\tag{by~\eqref{eq:pre-image_normalization_mapping}
since $0\in\LevelSet{\lzero}{k}$ }
\\
&=
\delta_{ \LevelSet{\lzero}{k}  } 
\tag{ as \( \lzero \circ \normalized = \lzero \) 
by~\eqref{eq:pseudonormlzero_invariance_normalized} }
 \eqfinp
\end{align*}
\medskip

\noindent$\bullet$ We prove~\eqref{eq:conjugate_l0norm}:
\begin{align*}
    \SFM{ \lzero }{\couplingCAPRA}
  &= 
\SFM{\Bp{ \inf_{l=0,1,\ldots,d} \ba{ \delta_{ \LevelCurve{\lzero}{l} } \UppPlus
    l } } }{\couplingCAPRA}
\tag{ since \(  \lzero = 
\inf_{l=0,1,\ldots,d} \ba{ \delta_{ \LevelCurve{\lzero}{l} } \UppPlus l } \)
by using the level curves~\eqref{eq:level_curve_pseudonormlzero} }
\\
  &= 
\sup_{l=0,1,\ldots,d} \ba{ 
\SFM{ \delta_{ \LevelCurve{\lzero}{k} }}{\couplingCAPRA}
 \LowPlus \np{-l} }  
\tag{ as conjugacies, being dualities, turn infima into suprema}
\\
  &= 
\sup_{l=0,1,\ldots,d} \ba{ 
\sigma_{ \normalized\np{\LevelCurve{\lzero}{l} } } 
 \LowPlus \np{-l} }  
\tag{ as \( \SFM{ \delta_{ \LevelCurve{\lzero}{k} }}{\couplingCAPRA}
= \sigma_{ \normalized\np{\LevelCurve{\lzero}{l} } } \)
by~\eqref{eq:one-sided_linear_Fenchel-Moreau_characteristic}}
\\
  &= 
\sup \Ba{ 0, \sup_{l=1,\ldots,d} \ba{ 
\sigma_{ \SPHERE \cap \LevelCurve{\lzero}{l} } 
 \LowPlus \np{-l} } }
\tag{ as \( \normalized\np{\LevelCurve{\lzero}{l} }
= \SPHERE \cap \LevelCurve{\lzero}{l} \) when $l\geq 1$ 
by~\eqref{eq:normalization_mapping} }
\\
  &= 
\sup \Ba{ 0, \sup_{l=1,\ldots,d} \ba{ 
\sigma_{ \overline{\SPHERE \cap \LevelCurve{\lzero}{l}} } 
 \LowPlus \np{-l} } }
\tag{ as \( \sigma_{ \Primal } = \sigma_{ \overline{\Primal} } \) 
for any \( \Primal \subset \PRIMAL \) }
\\
  &= 
\sup \Ba{ 0, \sup_{l=1,\ldots,d} \ba{ 
\sigma_{ \SPHERE \cap \LevelSet{\lzero}{l} } \LowPlus \np{-l} } }
\tag{ as \( \overline{\SPHERE \cap \LevelCurve{\lzero}{l}} =
\SPHERE \cap \LevelSet{\lzero}{l} \) by~\eqref{eq:closure_level_curve} }
\\
  &= 
\sup \Ba{ 0, \sup_{l=1,\ldots,d} \ba{ 
\sigma_{ \cup_{ \cardinal{K} \leq k} \SPHERE_{K} } \LowPlus \np{-l} } }
\tag{ as \( \SPHERE \cap \LevelSet{\lzero}{l} 
= \cup_{ \cardinal{K} \leq k} \SPHERE_{K}  \) 
by~\eqref{eq:level_set_l0_inter_sphere_a} } 
  \\
  &= 
    \sup \Ba{0, 
    \sup_{ l=1,\ldots,d} 
    \Bc{ \SymmetricGaugeNorm{l}{\dual} -l } }
\tag{ as \( \sup_{ \cardinal{K} \leq k} \sigma_{ \SPHERE_{K} }
= \SymmetricGaugeNorm{k}{\cdot} \) by~\eqref{eq:k-norm-from-support}}
  \\
  &= 
     \sup_{ l=0,1,\ldots,d} 
    \Bc{ \SymmetricGaugeNorm{l}{\dual} -l } 
\tag{ with the convention that \( \SymmetricGaugeNorm{0}{\cdot} = 0 \) }
\eqfinp
\end{align*}

\noindent$\bullet$ 
We prove~\eqref{eq:biconjugate_l0norm}.

It is easy to check that 
\( \SFMbi{ \lzero }{\couplingCAPRA}\np{0}=0=\lzero\np{0} \). 
Therefore, let $\primal \in \RR^d\backslash\{0\} $ be given and assume that 
$\lzero\np{\primal}=l \in \ba{1,\ldots,d}$. 
We consider the mapping $\phi: ]0,+\infty[ \to \RR$ defined by
\begin{equation}
\phi(\lambda) =
    \frac{ \proscal{\primal}{ \lambda \primal} }{ \norm{\primal} }
  + \Bp{ - \sup \Ba{0, \sup_{j=1,\ldots,d} 
\Bc{ \SymmetricGaugeNorm{j}{\lambda\primal} - j
    } } }
\eqsepv \forall \lambda > 0 \eqfinv
  \label{eq:phi}
\end{equation}
and we will show that 
\( \lim_{\lambda \to +\infty} \phi(\lambda) =l \).
We have 
\begin{align*}
  \phi(\lambda) 
  &=
    \lambda \norm{\primal}
    + \Bp{ - \sup \Ba{0, \sup_{j=1,\ldots,d} 
\Bc{ \SymmetricGaugeNorm{j}{\lambda\primal} - j
    } } }
\tag{ by definition~\eqref{eq:phi} of~$\phi$ }
\\
  &= 
    \lambda \SymmetricGaugeNorm{l}{\primal}
    + \inf \Ba{0, -\sup_{j=1,\ldots,d} 
\Bc{ \lambda \SymmetricGaugeNorm{j}{\primal} - j } }
    \tag{ as $\norm{\primal}=\SymmetricGaugeNorm{l}{\primal}$ 
when $\lzero\np{\primal}=l$
by~\eqref{eq:level_curve_l0_characterization} } 
\\
  &= 
\inf \Ba{ \lambda \SymmetricGaugeNorm{l}{\primal} ,
\lambda \SymmetricGaugeNorm{l}{\primal} + 
\inf_{j=1,\ldots,d} \Bp{ -\Bc{ \lambda \SymmetricGaugeNorm{j}{\primal} - j } } }
\\
  &= 
\inf \Ba{ \lambda \SymmetricGaugeNorm{l}{\primal} ,
\inf_{j=1,\ldots,d} \Bp{ \lambda \bp{ \SymmetricGaugeNorm{l}{\primal} -
    \SymmetricGaugeNorm{j}{\primal} } + j } }
\\
  &= 
\inf \Ba{ \lambda \SymmetricGaugeNorm{l}{\primal} ,
\inf_{j=1,\ldots,l} \Bp{ \lambda \bp{ \SymmetricGaugeNorm{l}{\primal} -
    \SymmetricGaugeNorm{j}{\primal} } + j } }
    \tag{ as $\SymmetricGaugeNorm{j}{\primal}= \SymmetricGaugeNorm{l}{\primal}$ for
    $j\ge l$ by~\eqref{eq:level_curve_l0_characterization} }
\\
  &= 
\inf \Ba{ \lambda \SymmetricGaugeNorm{l}{\primal} ,
\inf_{j=1,\ldots,l-1} \Bp{ \lambda \bp{ \SymmetricGaugeNorm{l}{\primal} -
    \SymmetricGaugeNorm{j}{\primal} } + j } , l }
\eqfinp
\end{align*}
Let us show that the two first terms in the infimum
go to $+\infty$ when \( \lambda \to +\infty \).
The first term goes to $+\infty$ 
because \( \SymmetricGaugeNorm{l}{\primal}=\norm{\primal}>0 \) by assumption
($\primal \neq 0$).
The second term also goes to $+\infty$ because 
$\lzero\np{\primal}=l $, so that 
\( \norm{\primal}=\SymmetricGaugeNorm{l}{\primal} > 
\SymmetricGaugeNorm{j}{\primal} \)
for $j=1,\ldots,l-1$
by~\eqref{eq:level_curve_l0_characterization}.
Therefore, 
\( \lim_{\lambda \to +\infty} \phi(\lambda) = 
\inf \{ +\infty, +\infty, l \}=l \).
This concludes the proof since 
\begin{align*}
  l = \lim_{\lambda \to +\infty} \phi(\lambda) 
& \leq 
\sup_{\dual \in \RR^d } \bgp{ \frac{ \proscal{\primal}{\dual} }{ \norm{\primal} }
      \LowPlus 
\Bp{ - \sup \Ba{0, \sup_{j=1,\ldots,d} 
\Bc{ \SymmetricGaugeNorm{j}{\dual} - j } } } }
\tag{ by definition~\eqref{eq:phi} of~$\phi$ }
\\
&=
\sup_{\dual \in \RR^d } \bgp{ \frac{ \proscal{\primal}{\dual} }{ \norm{\primal} }
      \LowPlus 
\Bp{ - \sup_{j=0,1,\ldots,d} 
\Bc{ \SymmetricGaugeNorm{j}{\dual} - j } } }
\tag{ by the convention that \( \SymmetricGaugeNorm{0}{\cdot} = 0 \) }
\\
&=
\sup_{\dual \in \RR^d } \bgp{ \frac{ \proscal{\primal}{\dual} }{ \norm{\primal} }
      \LowPlus \Bp{ - \SFM{ \lzero }{\couplingCAPRA}\np{\dual} } }
\tag{ by the formula~\eqref{eq:conjugate_l0norm} for 
\( \SFM{ \lzero }{\couplingCAPRA} \) }
\\
&=
\SFMbi{ \lzero }{\couplingCAPRA}\np{\primal}
\tag{ by the biconjugate formula~\eqref{eq:Fenchel-Moreau_biconjugate} }
\\
& \leq 
    \lzero\np{\primal} 
\tag{ by~\eqref{eq:galois-cor} giving 
\( \SFMbi{ \lzero }{\couplingCAPRA} \leq \lzero \)}
\\
& = l 
\tag{ by assumption }
\eqfinp
\end{align*}
Therefore, we have obtained \( l=\SFMbi{ \lzero
}{\couplingCAPRA}\np{\primal}=\lzero\np{\primal}\).

This ends the proof.
\end{proof}

\begin{corollary}
The \pseudonormlzero~$\lzero$ coincides, 
on the sphere~$\SPHERE$, 
with a convex lsc function defined on the whole space~$\RR^d$:
\begin{equation}
\lzero\np{\primal} =
\LFM{ \Bp{ \sup_{l=0,1,\ldots,d} \Bc{ \SymmetricGaugeNorm{l}{\cdot} -l } } }
\np{\primal} 
\eqsepv 
\forall \primal \in \SPHERE 
\eqfinp 
\end{equation}
\end{corollary}

\begin{proof}
For \( \primal \in \SPHERE \), we have 
\begin{align*}
  \lzero\np{\primal} 
  &=
    \SFMbi{ \lzero }{\couplingCAPRA}\np{\primal}
\tag{ by~\eqref{eq:biconjugate_l0norm} }
\\
&=
   \sup_{ \dual \in \RR^d } \Bp{ \couplingCAPRA\np{\primal,\dual} 
      \LowPlus \bp{ -\SFM{ \lzero }{\couplingCAPRA}\np{\dual} } } 
\tag{ by the biconjugate formula~\eqref{eq:Fenchel-Moreau_biconjugate} }
\\
&=
   \sup_{ \dual \in \RR^d } \Bp{ \proscal{\primal}{\dual} 
      \LowPlus \bp{ -\SFM{ \lzero }{\couplingCAPRA}\np{\dual} } } 
\tag{ by~\eqref{eq:coupling_CAPRAC_original} with
\( \norm{\primal}=1 \) since \( \primal \in \SPHERE \) } 
\\
&=
   \sup_{ \dual \in \RR^d } \bgp{ \proscal{\primal}{\dual} 
      \LowPlus \Bp{ -\bp{ 
\sup_{l=0,1,\ldots,d} \Bc{ \SymmetricGaugeNorm{l}{\dual} -l }
} }
  }
\tag{ by~\eqref{eq:conjugate_l0norm} }
\\
&=
\SFM{ \Bp{ 
\sup_{l=0,1,\ldots,d} \Bc{ \SymmetricGaugeNorm{l}{\dual} -l }
} }{\star}\np{\primal}  
\tag{by the expression~\eqref{eq:Fenchel_conjugate} of the Fenchel conjugate}
\eqfinp 
\end{align*}
This ends the proof.
\end{proof}

\section{Conclusion}

In this paper, we have introduced a novel class of 
one-sided linear couplings, and have shown that they induce conjugacies 
that share nice properties with the classic Fenchel conjugacy.
Among them, we have distinguished a novel coupling, \Capra, 
having the property of being constant along primal rays, 
like the \pseudonormlzero.
For the \Capra\ conjugacy, induced by the coupling \Capra, 
we have proved that the \pseudonormlzero\ is equal to its biconjugate:
hence, the \pseudonormlzero\ is \Capra-convex 
in the sense of generalized convexity.
We have also provided expressions for conjugates in terms 
of two families of dual norms, the $2$-$k$-symmetric gauge norms
and the $k$-support norms. 

In a companion paper \cite{Chancelier-DeLara:2019c},
we apply our results to so-called sparse optimization, that is, 
problems where one looks for solution that have few nonzero components.
We provide a systematic way to obtain 
convex minimization programs (over unit balls of some norms)
that are lower bounds
for the original exact sparse optimization problem. 
\bigskip

\textbf{Acknowledgements.}
We want to thank 
Juan Enrique Martínez Legaz
and
Jean-Baptiste Hiriart-Urruty
for discussions on first versions of this work. 

\appendix

\section{Appendix}
\label{Appendix}

\subsection{Background on J.~J. Moreau lower and upper additions}
\label{Moreau_lower_and_upper_additions}




When we manipulate functions with values 
in~$\bar\RR = [-\infty,+\infty] $,
we adopt the following Moreau \emph{lower addition} or
\emph{upper addition}, depending on whether we deal with $\sup$ or $\inf$
operations. 
We follow \cite{Moreau:1970}.
In the sequel, $u$, $v$ and $w$ are any elements of~$\bar\RR$.

\subsubsection*{Moreau lower addition}

\begin{subequations}
  The Moreau \emph{lower addition} extends the usual addition with 
  \begin{equation}
    \np{+\infty} \LowPlus \np{-\infty} = \np{-\infty} \LowPlus \np{+\infty} = -\infty \eqfinp
    \label{eq:lower_addition}
  \end{equation}
  With the \emph{lower addition}, \( \np{\bar\RR, \LowPlus } \) is a convex cone,
  with $ \LowPlus $ commutative and associative.
  The lower addition displays the following properties:
  \begin{align}
    u \leq u' \eqsepv v \leq v' 
    & 
      \Rightarrow u \LowPlus v \leq u'  \LowPlus v' 
\eqfinv 
\label{eq:lower_addition_leq}
\\
    (-u) \LowPlus (-v) 
    & 
      \leq -(u \LowPlus v) 
\eqfinv 
\label{eq:lower_addition_minus}
\\ 
    (-u) \LowPlus u 
    & 
      \leq 0 
\eqfinv 
\label{eq:lower_substraction_le_zero}
\\
    \sup_{a\in\AAA} f(a) \LowPlus \sup_{b\in\BB} g(b)
    &=
      \sup_{a\in\AAA,b\in\BB} \bp{f(a) \LowPlus g(b)} 
\eqfinv
\label{eq:lower_addition_sup}
\\ 
    \inf_{a\in\AAA} f(a) \LowPlus \inf_{b\in\BB} g(b) 
    & \leq 
      \inf_{a\in\AAA,b\in\BB} \bp{f(a) \LowPlus g(b)} 
\eqfinv
\label{eq:lower_addition_inf}
\\ 
    t < +\infty \Rightarrow    \inf_{a\in\AAA} f(a) \LowPlus t 
    & =
      \inf_{a\in\AAA} \bp{f(a) \LowPlus t} 
\eqfinp
\label{eq:lower_addition_inf_constant}
  \end{align}
\end{subequations}

\subsubsection*{Moreau upper addition}

\begin{subequations}
  The Moreau \emph{upper addition} extends the usual addition with 
  \begin{equation}
    \np{+\infty} \UppPlus \np{-\infty} = 
    \np{-\infty} \UppPlus \np{+\infty} = +\infty \eqfinp
    \label{eq:upper_addition}
  \end{equation}
  With the \emph{upper addition}, \( \np{\bar\RR, \UppPlus } \) is a convex cone,
  with $ \UppPlus $ commutative and associative.
  The upper addition displays the following properties:
  \begin{align}
    u \leq u' \eqsepv v \leq v' 
    & 
      \Rightarrow u \UppPlus v \leq u' \UppPlus v' \eqfinv
      \label{eq:upper_addition_leq} \\
    (-u) \UppPlus (-v) 
    & 
      \geq -(u \UppPlus v) \eqfinv
      \label{eq:upper_addition_minus} 
    \\ 
    (-u) \UppPlus u 
    & 
      \geq 0 \eqfinv
      \label{eq:upper_substraction_ge_zero} 
    \\ 
    \inf_{a\in\AAA} f(a) \UppPlus \inf_{b\in\BB} g(b) 
    & =
      \inf_{a\in\AAA,b\in\BB} \bp{f(a) \UppPlus g(b)} \eqfinv
      \label{eq:upper_addition_inf} 
    \\
    \sup_{a\in\AAA} f(a) \UppPlus \sup_{b\in\BB} g(b)
    & \geq 
      \sup_{a\in\AAA,b\in\BB} \bp{f(a) \UppPlus g(b)} \eqfinv
      \label{eq:upper_addition_sup}
    \\ 
    -\infty < t \Rightarrow    \sup_{a\in\AAA} f(a) \UppPlus t 
    & =
      \sup_{a\in\AAA} \bp{f(a) \UppPlus t} \eqfinp
      \label{eq:upper_addition_sup_constant}
  \end{align}
\end{subequations}

\subsubsection*{Joint properties of the Moreau lower and upper addition}

\begin{subequations}
We obviously have that
\begin{equation}
  u \LowPlus v \leq u \UppPlus v \eqfinp 
\label{eq:lower_leq_upper_addition}
\end{equation}
The Moreau lower and upper additions are related by
\begin{equation}
-(u \UppPlus v) = (-u) \LowPlus (-v) \eqsepv 
-(u \LowPlus v) = (-u) \UppPlus (-v) \eqfinp 
\label{eq:lower_upper_addition_minus}  
\end{equation}
They satisfy the inequality
\begin{equation}
(u \UppPlus v) \LowPlus w \leq u \UppPlus (v \LowPlus w) \eqfinp
                                 \label{eq:lower_upper_addition_inequality} 
\end{equation}
with 
  \begin{equation}
(u \UppPlus v) \LowPlus w < u \UppPlus (v \LowPlus w) 
\iff
\begin{cases}
 u=+\infty \mtext{ and } w=-\infty \eqsepv \\
\mtext{ or } \\
u=-\infty \mtext{ and } w=+\infty 
\mtext{ and } -\infty < v < +\infty \eqfinp
\end{cases}
                                 \label{eq:lower_upper_addition_equality} 
\end{equation}
Finally, we have that 
\begin{align}
& u \LowPlus (-v) \leq 0 \iff u \leq v  \iff  0 \leq v \UppPlus (-u) 
\eqfinv
\label{eq:lower_upper_addition_comparisons}
\\
& u \LowPlus (-v) \leq w \iff u \leq v \UppPlus w \iff u \LowPlus (-w) \leq v
\eqfinv 
\\  
& w \leq v \UppPlus (-u) \iff u \LowPlus w \leq v \iff u \leq v \UppPlus (-w) 
\eqfinp
\end{align}
 \end{subequations}



\subsection{Properties of $2$-$k$-symmetric gauge norms}

Before studying properties of $2$-$k$-symmetric gauge norms,
we recall the notion of dual norm.
\medskip

Let \( \triplenorm{\cdot} \) be a norm on~$\RR^d$,
with unit ball denoted by
\begin{equation}
    \BALL_{\triplenorm{\cdot}}
= 
\defset{\primal \in \RR^d}{\triplenorm{\primal} \leq 1} 
 \eqfinp
    \label{eq:triplenorm_unit_ball}
\end{equation}

\begin{definition}
The following expression 
  \begin{equation}
    \triplenorm{\dual}_\star = 
    \sup_{ \triplenorm{\primal} \leq 1 } \proscal{\primal}{\dual} 
    \eqsepv \forall \dual \in \DUAL
  \end{equation}
defines a norm on~$\DUAL$, 
called the \emph{dual norm} \( \triplenorm{\cdot}_\star \).
\label{de:dual_norm}
\end{definition}
We have 
\begin{equation}
    \triplenorm{\cdot}_\star = \sigma_{\BALL_{\triplenorm{\cdot}}} 
\mtext{ and } 
\triplenorm{\cdot} = \sigma_{\BALL_{\triplenorm{\cdot}_\star}}
\eqfinv
    \label{eq:norm_dual_norm}
\end{equation}
where \( \BALL_{\triplenorm{\cdot}_\star} \) is 
the unit ball of the dual norm:
  \begin{equation}
\BALL_{\triplenorm{\cdot}_\star}
= \defset{\dual \in \DUAL}{\triplenorm{\dual}_\star \leq 1}
     \eqfinp
  \end{equation}

For all \( K \subset\ba{1,\ldots,d} \), 
we introduce \emph{degenerate} unit ``spheres'' and ``balls'' of \( \RR^d \) by 
\begin{subequations}
  \begin{align}
    \ESPHERE_{K} 
    &= 
\defset{\primal \in \RR^d}{ \norm{\primal_{K}} = 1} 
\eqfinv 
\label{eq:ESPHERE_K}
\\
    \EBALL_{K} 
    &=
 \defset{\primal \in \RR^d}{ \norm{\primal_{K}} \le 1} 
      \eqfinv 
\label{eq:EBALL_K}
\\
    \SPHERE_{K} 
    &= 
\defset{\primal \in \RR^d}%
      {\primal_{-K}=0 \mtext{ and } \norm{\primal_{K}} = 1} 
      \eqfinv 
\label{eq:SPHERE_K} 
\\
    \BALL_{K} 
    &= 
\defset{\primal \in \RR^d}%
      {\primal_{-K}=0 \mtext{ and } \norm{\primal_{K}} \leq 1} 
      \eqfinv    
\label{eq:BALL_K} 
  \end{align}
\end{subequations}
where $\primal_{K}$ has been defined right before
Definition~\ref{de:symmetric_gauge_norm}.

In what follows, the notation \( \bigcup_{\cardinal{K} \leq k} \)
is a shorthand for 
\( \bigcup_{ { K \subset \na{1,\ldots,d}, \cardinal{K} \leq k}} \),
\( \bigcap_{\cardinal{K} \leq k} \) for 
\( \bigcap_{ { K \subset \na{1,\ldots,d}, \cardinal{K} \leq k}} \),
and \( \sup_{\cardinal{K} \leq k} \) for 
\( \sup_{ { K \subset \na{1,\ldots,d}, \cardinal{K} \leq k}} \)). 
The same holds true for \( \bigcup_{\cardinal{K} = k} \), 
\( \bigcap_{\cardinal{K} = k} \)
and \( \sup_{\cardinal{K} = k} \).

\begin{proposition}
Let \( k \in \ba{1,\ldots,d} \).
\begin{itemize}
\item 
The following inequalities hold true
  \begin{equation}
    \sup_{ j=1,\ldots,d } \module{\primal_j} = 
\norm{\primal}_{\infty} =  \SymmetricGaugeNorm{1}{\primal} \leq \cdots \leq 
    \SymmetricGaugeNorm{k}{\primal} \leq 
\SymmetricGaugeNorm{k+1}{\primal} \leq \cdots \leq 
    \SymmetricGaugeNorm{n}{\primal} =  \norm{\primal} 
    \eqfinp
    \label{eq:k_approx_props}
  \end{equation}
\item 
The two ``spheres'' in~\eqref{eq:ESPHERE_K} and \eqref{eq:SPHERE_K}
are related by
  \begin{equation}
    \SPHERE_{K} = \SPHERE \cap \ESPHERE_{K}  
\eqsepv \forall K \subset \ba{1,\ldots,d} 
\eqfinp
    \label{eq:sphere_and_esphere_K}
  \end{equation}
\item 
The $2$-$k$-symmetric gauge norm~\( \SymmetricGaugeNorm{k}{\cdot} \) 
in Definition~\ref{de:symmetric_gauge_norm} satisfies
  \begin{equation}
    \SymmetricGaugeNorm{k}{\cdot} = 
    \sigma_{ \cup_{ \cardinal{K} \leq k} \BALL_{K} }
= \sup_{ \cardinal{K} \leq k } \sigma_{ \BALL_{K} }
= \sup_{ \cardinal{K} \leq k } \sigma_{ \SPHERE_{K} }
= \sigma_{ \cup_{ \cardinal{K} \leq k} \SPHERE_{K} } 
    \eqfinv
    \label{eq:k-norm-from-support}
  \end{equation}
where \( \cardinal{K} \leq k \) can be replaced by \( \cardinal{K} = k \) 
everywhere.
\item 
The unit sphere~\( \SymmetricGauge{k}{\SPHERE} \) 
and ball~\( \SymmetricGauge{k}{\BALL} \)
of \( \RR^d \) for the $2$-$k$-symmetric gauge norm
\( \SymmetricGaugeNorm{k}{\cdot} \) 
in Definition~\ref{de:symmetric_gauge_norm} satisfy 
  \begin{subequations}
    \begin{align}
      \SymmetricGauge{k}{\BALL} &= 
\defset{\primal \in \RR^d}{ \SymmetricGaugeNorm{k}{\primal} \leq 1} 
                    = \bigcap_{\cardinal{K} \leq k} \EBALL_{K}
                    \eqfinv    
                    \label{eq:symmetric_gauge_norm_unit-ball} \\
      \SymmetricGauge{k}{\SPHERE} 
&= 
\defset{\primal \in \RR^d}{ \SymmetricGaugeNorm{k}{\primal} = 1} 
                      = \SymmetricGauge{k}{\BALL} \cap
         \Bp{ \bigcup_{\cardinal{K} \leq k} \ESPHERE_{K} }
                      \eqfinv
                      \label{eq:symmetric_gauge_norm_unit-sphere}
    \end{align}
  \end{subequations}
where \( \cardinal{K} \leq k \) can be replaced by \( \cardinal{K} = k \) everywhere.
\item 
The unit ball~\( \SupportBall{k}{\BALL} \) of the 
$k$-support norm~\( \SupportNorm{k}{\cdot} \)
in Definition~\ref{de:symmetric_gauge_norm} satisfies
\begin{equation}
  \SupportBall{k}{\BALL} = 
  \defset{\primal \in \RR^d}{ \SupportNorm{k}{\primal} \leq 1} 
= \closedconvexhull\bp{ \bigcup_{ \cardinal{K} \leq k} \BALL_{K} }
= \closedconvexhull\bp{ \bigcup_{ \cardinal{K} \leq k} \SPHERE_{K} }
\eqfinv
\label{eq:dual_support_norm_unit_ball}
\end{equation}
where \( \cardinal{K} \leq k \) can be replaced by \( \cardinal{K} = k \) 
everywhere.
\end{itemize}
\end{proposition}

\begin{proof}

\noindent $\bullet$ 
The Inequalities and Equalities~\eqref{eq:k_approx_props} easily derive 
from the very definition~\eqref{eq:symmetric_gauge_norm}
of the $2$-$k$-symmetric gauge norm
\( \SymmetricGaugeNorm{k}{\cdot} \).
\medskip

\noindent $\bullet$ 
We prove Equation~\eqref{eq:sphere_and_esphere_K}.
Recall that, following notation from Game Theory, 
we denote by $-K$ the complementary subset 
of~$K$ in \( \ba{1,\ldots,d} \): \( K \cup (-K) = \ba{1,\ldots,d} \)
and \( K \cap (-K) = \emptyset \). 
Then, we have that \( \primal=\primal_K+\primal_{-K} \),
for any \( \primal\in\RR^d \),  
and the decomposition is orthogonal, leading to
\begin{equation}
\bp{\forall \primal \in \RR^d } \qquad
  \primal=\primal_K+\primal_{-K} 
\eqsepv \primal_K \perp \primal_{-K} 
\mtext{ and }
\norm{\primal}^2=\norm{\primal_K}^2+\norm{\primal_{-K}}^2
\eqfinp
\label{eq:decomposition_orthogonal}
\end{equation}
For \( K \subset \ba{1,\ldots,d} \), we have that 
  \begin{align*}
    \primal \in \SPHERE \mtext{ and } 
\primal \in \ESPHERE_{K}  
    &\iff 
      1=\norm{\primal}^2 \mtext{ and } 1=\norm{\primal_K}^2 
\tag{ by~\eqref{eq:ESPHERE_K} }
    \\
    &\iff 
      1=\norm{\primal}^2=\norm{\primal_K}^2+\norm{\primal_{-K}}^2
      \mtext{ and } 1=\norm{\primal_K}^2
\tag{ by~\eqref{eq:decomposition_orthogonal} }
\\
    &\iff 
      \norm{\primal_{-K}}=0 \mtext{ and } 1=\norm{\primal_K} 
\tag{ by~\eqref{eq:decomposition_orthogonal} }
    \\
    &\iff 
      \primal \in \SPHERE_{K} 
\tag{ by~\eqref{eq:SPHERE_K}}
\eqfinp 
  \end{align*}
\medskip

\noindent $\bullet$ 
We prove Equation~\eqref{eq:k-norm-from-support}. 
For this purpose, we first establish that
  \begin{equation}
    \sigma_{ \BALL_{K} }(\dual) 
=\norm{\dual_K} 
    \eqsepv \forall \dual \in \RR^d 
    \eqfinp   
    \label{eq:support_functions_unit_sphere_ball_K}
  \end{equation}
Indeed, for \( \dual \in \RR^d \), we have
\begin{align*}
      \sigma_{ \BALL_{K} }(\dual)
&=
\sup_{ \primal \in \BALL_{K} } \proscal{\primal}{\dual}
\tag{ by definition~\eqref{eq:support_function} of a support function }
\\
&=
\sup_{ \primal \in \BALL_{K} } 
\proscal{\primal_K+\primal_{-K}}{\dual_K+\dual_{-K}}
\tag{ by the decomposition~\eqref{eq:decomposition_orthogonal} }
\\
&=
\sup_{ \primal \in \BALL_{K} } 
\bp{ \proscal{\primal_K}{\dual_K}
+ \proscal{\primal_{-K}}{\dual_{-K}} }
\tag{ because \( \primal_K \perp \dual_{-K} \) and
\( \primal_{-K} \perp \dual_K \) by~\eqref{eq:decomposition_orthogonal} }
\\
&=
\sup \ba{  \proscal{\primal_K}{\dual_K}
+ \proscal{\primal_{-K}}{\dual_{-K}}, \, 
\primal_{-K}=0 \mtext{ and } \norm{\primal_{K}} \leq 1 } 
\tag{ by definition~\eqref{eq:BALL_K} of \( \BALL_{K} \) }
\\
&=
\sup \ba{  \proscal{\primal_K}{\dual_K}, \, 
\norm{\primal_{K}} \leq 1 } 
\\
&=
\norm{\dual_K} 
\end{align*}
as is well-known for the Euclidian norm~$\norm{\cdot}$, when restricted 
to the subspace \( \defset{\primal \in \RR^d}{\primal_{-K}=0} \)
(because it is equal to its dual norm). 
Then, for all $\dual \in \RR^d$, we have that 
\begin{align*}
  \sigma_{ \cup_{ \cardinal{K} \leq k } \BALL_{K} }(\dual)
  &= 
\sup_{\cardinal{K} \leq k} \sigma_{ \BALL_{K} }(\dual) 
\tag{ as is well-known in convex analysis }
\\
  &= 
\sup_{\cardinal{K} \leq k} \norm{\dual_K}
    \tag{by~\eqref{eq:support_functions_unit_sphere_ball_K} }
  \\
  & =\SymmetricGaugeNorm{k}{\dual} 
\tag{by definition~\eqref{eq:symmetric_gauge_norm}
of \( \SymmetricGaugeNorm{k}{\cdot} \) }
 \eqfinp
\end{align*}
Now, by~\eqref{eq:SPHERE_K} and \eqref{eq:BALL_K}, 
it is straightforward that 
\( \closedconvexhull\np{\SPHERE_{K}} = \BALL_{K} \)
and we deduce that 
\[
\SymmetricGaugeNorm{k}{\cdot} 
=
\sigma_{ \cup_{ \cardinal{K} \leq k } \BALL_{K} }
=
\sup_{\cardinal{K} \leq k} \sigma_{ \BALL_{K} }
=
\sup_{\cardinal{K} \leq k} 
\sigma_{ \closedconvexhull\np{\SPHERE_{K}} }
=
\sup_{\cardinal{K} \leq k} \sigma_{ \SPHERE_{K} }
=
\sigma_{ \cup_{ \cardinal{K} \leq k } \SPHERE_{K} }
\eqfinv
\]
giving Equation~\eqref{eq:k-norm-from-support}. 

If we take over the proof using the property that
 \( \sup_{\cardinal{K} \leq k}\norm{\dual_K} 
= \sup_{\cardinal{K} = k}\norm{\dual_K} \) 
in~\eqref{eq:symmetric_gauge_norm},
we deduce that 
\( \cardinal{K} \leq k \) can be replaced by \( \cardinal{K} = k \) everywhere.
\medskip

\noindent $\bullet$ 
We prove Equation~\eqref{eq:symmetric_gauge_norm_unit-ball}:
  \begin{align*}
    \SymmetricGauge{k}{\BALL}  
    &= 
\defset{\primal \in \RR^d}{ \SymmetricGaugeNorm{k}{\primal} \leq 1} 
\tag{ by definition of the ball \( \SymmetricGauge{k}{\BALL} \) }
\\
    &= 
\defset{\primal \in \RR^d}{ \sup_{ 
\cardinal{K} \leq k}  
      \norm{\primal_K} \leq 1} 
\tag{ by definition~\eqref{eq:symmetric_gauge_norm}
of \( \SymmetricGaugeNorm{k}{\cdot} \) }
\\
    &= 
\bigcap_{ 
\cardinal{K} \leq k} 
      \defset{\primal \in \RR^d}{ \norm{\primal_K} \leq 1} 
\\
    &= 
\bigcap_{ 
\cardinal{K} \leq k}
      \EBALL_{K}
\eqfinp
\tag{ by definition~\eqref{eq:EBALL_K} of \( \EBALL_{K} \) }
  \end{align*}
If we take over the proof using the property that
 \( \sup_{\cardinal{K} \leq k}\norm{\dual_K} 
= \sup_{\cardinal{K} = k}\norm{\dual_K} \) 
in~\eqref{eq:symmetric_gauge_norm},
we deduce that 
\( \cardinal{K} \leq k \) can be replaced by \( \cardinal{K} = k \) everywhere.
\medskip

\noindent $\bullet$ 
We prove Equation~\eqref{eq:symmetric_gauge_norm_unit-sphere}: 
  \begin{align*}
    \SymmetricGauge{k}{\SPHERE}  
    &= \defset{\primal \in \RR^d}{ \SymmetricGaugeNorm{k}{\primal} = 1} 
\tag{ by definition of the sphere $\SymmetricGauge{k}{\SPHERE}$ }
\\
    &= \defset{\primal \in \RR^d}{ \sup_{ 
      \cardinal{K} \leq k}  \norm{\primal_K} = 1} 
\tag{ by definition~\eqref{eq:symmetric_gauge_norm}
of \( \SymmetricGaugeNorm{k}{\cdot} \) }
\\
    &= \defset{\primal \in \RR^d}%
{ \sup_{ 
\cardinal{K} \leq k}  \norm{\primal_K} \le 1}\\
    & \hphantom{===} \bigcap 
      \defset{\primal \in \RR^d}%
{ \exists K \subset \ba{1,\ldots,d} \eqsepv \cardinal{K} \leq k 
\eqsepv \norm{\primal_K} = 1 }  
\\
    &= \SymmetricGauge{k}{\BALL} \cap 
      \Bp{\bigcup_{ 
\cardinal{K} \leq k} 
      \defset{\primal \in \RR^d}{ \norm{\primal_K} = 1}} 
\tag{ by definition of the ball \( \SymmetricGauge{k}{\BALL} \) }
\\
    &= \SymmetricGauge{k}{\BALL} \cap 
      \Bp{
      \bigcup_{ 
\cardinal{K} \leq k} 
      \ESPHERE_{K}}
\eqfinp
\tag{ by definition~\eqref{eq:ESPHERE_K} of \( \ESPHERE_{K} \) }
  \end{align*}
If we take over the proof using the property that
 \( \sup_{\cardinal{K} \leq k}\norm{\dual_K} 
= \sup_{\cardinal{K} = k}\norm{\dual_K} \) 
in~\eqref{eq:symmetric_gauge_norm},
we deduce that 
\( \cardinal{K} \leq k \) can be replaced by \( \cardinal{K} = k \) 
everywhere.
\medskip

\noindent $\bullet$ 
We prove Equation~\eqref{eq:dual_support_norm_unit_ball}.
On the one hand, by the first relation in~\eqref{eq:norm_dual_norm},
we have that \( \SymmetricGaugeNorm{k}{\cdot}=\sigma_{ \SupportBall{k}{\BALL} } \).
On the other hand, 
by~\eqref{eq:k-norm-from-support}, we have that 
\( \SymmetricGaugeNorm{k}{\cdot}
=
\sigma_{ \cup_{\cardinal{K} \leq k} \BALL_{K} } 
=
\sigma_{ \cup_{\cardinal{K} \leq k} \SPHERE_{K} } \).
Then, as is well-known in convex analysis, we deduce that
\( 
\closedconvexhull\bp{ \SupportBall{k}{\BALL} } 
=
\closedconvexhull\bp{ \bigcup_{ \cardinal{K} \leq k} \BALL_{K} }
=
\closedconvexhull\bp{ \bigcup_{ \cardinal{K} \leq k} \SPHERE_{K} }
\).
As the unit ball~\( \SupportBall{k}{\BALL} \) is closed and convex,
we immediately obtain~\eqref{eq:dual_support_norm_unit_ball}.

If we take over the proof using the property that
 \( \sigma_{ \cup_{\cardinal{K} \leq k} \BALL_{K} } 
= \sigma_{ \cup_{\cardinal{K} \leq k} \SPHERE_{K} } 
= \sigma_{ \cup_{\cardinal{K} = k} \BALL_{K} } 
= \sigma_{ \cup_{\cardinal{K} = k} \SPHERE_{K} } \)
in~\eqref{eq:k-norm-from-support},
we deduce that 
\( \cardinal{K} \leq k \) can be replaced by \( \cardinal{K} = k \) everywhere.
\end{proof}

\subsection{Properties of the level sets of the \pseudonormlzero}

A connection between the \pseudonormlzero\ in~\eqref{eq:pseudo_norm_l0} 
and the $2$-$k$-symmetric gauge norm
\( \SymmetricGaugeNorm{k}{\cdot} \) in~\eqref{eq:symmetric_gauge_norm}
is given by the (easily proved) following Proposition.

\begin{proposition}
Let \( k \in \ba{0,1,\ldots,d} \).
For any \( \primal\in\RR^d \), we have 
\begin{equation}
\lzero\np{\primal} = k 
\iff 
0=\SymmetricGaugeNorm{0}{\primal} \leq \cdots \leq 
\SymmetricGaugeNorm{k-1}{\primal} < 
    \SymmetricGaugeNorm{k}{\primal} = 
\cdots =
    \SymmetricGaugeNorm{n}{\primal} =  \norm{\primal} 
\eqfinv
    \label{eq:level_curve_l0_characterization}
    \end{equation}
from which we deduce the formula
    \begin{equation}
  \lzero\np{\primal} = \min \defset{ j \in \ba{0,1,\ldots,d} }%
{ \SymmetricGaugeNorm{j}{\primal} = \norm{\primal} }
\eqfinv
\end{equation}
with the convention that \( \SymmetricGaugeNorm{0}{\cdot} = 0 \). 
\end{proposition}

We prove the following Proposition about the level sets of the \pseudonormlzero.

\begin{proposition}
Let \( k \in \ba{0,1,\ldots,d} \).
The level set \( \LevelSet{\lzero}{k} \) 
in~\eqref{eq:level_set_pseudonormlzero}
of the \pseudonormlzero\ in~\eqref{eq:pseudo_norm_l0} satisfies
\begin{subequations}
\begin{align}
\bp{\forall \primal \in \RR^d } \qquad
& 
\primal \in \LevelSet{\lzero}{k} 
\iff
\lzero\np{\primal} \leq k
\iff 
    \SymmetricGaugeNorm{k}{\primal} = \norm{\primal} 
    \eqfinv
    \label{equiv_norm0}
\\
\bp{\forall \primal \in \RR^d } \qquad
& 
\primal \in \LevelSet{\lzero}{k}\backslash\{0\}
\iff
0<\lzero\np{\primal} \leq k \iff
\primal\neq 0 \mtext{ and }
\frac{\primal}{\norm{\primal}} \in \SPHERE \cap \SymmetricGauge{k}{\SPHERE}  
\eqfinv
    \label{equiv_norm0_bis}
  \end{align}
\end{subequations}
and its intersection with the sphere~$\SPHERE$ has the three following
expressions
\begin{subequations}
  \begin{align}
    \SPHERE \cap \LevelSet{\lzero}{k} 
    &=
      \bigcup_{ {\cardinal{K} \leq k}} \SPHERE_{K} 
      = \bigcup_{ {\cardinal{K} = k}} \SPHERE_{K} 
      \eqfinv
      \label{eq:level_set_l0_inter_sphere_a}
    \\
    \SPHERE \cap \LevelSet{\lzero}{k} 
    &=
      \SPHERE \cap \SupportBall{k}{\BALL} 
      \eqfinv
      \label{eq:level_set_l0_inter_sphere_b}
    \\
    \SPHERE \cap \LevelSet{\lzero}{k} 
    &=
      \overline{ \SPHERE \cap \LevelCurve{\lzero}{k} }
      \eqfinp
      \label{eq:closure_level_curve}
  \end{align}
\end{subequations}
\end{proposition}

\begin{proof}

\noindent $\bullet$ 
The Equivalence~\eqref{equiv_norm0} easily follows 
from~\eqref{eq:level_curve_l0_characterization}.
\medskip

\noindent $\bullet$ 
We prove the Equivalence~\eqref{equiv_norm0_bis}.
Indeed, using Equation~\eqref{equiv_norm0} we have that,
for \( \primal \in \RR^d\backslash\{0\} \):
\begin{align*}
  \lzero\np{\primal} \leq k 
  &  \iff \SymmetricGaugeNorm{k}{\primal} = \norm{\primal} 
    \iff \SymmetricGaugeNorm{k}{\frac{\primal}{\norm{\primal}}} = 1 
    \iff \frac{\primal}{\norm{\primal}} \in \SymmetricGauge{k}{\SPHERE} 
    \iff \frac{\primal}{\norm{\primal}} \in 
\SPHERE \cap \SymmetricGauge{k}{\SPHERE}  
\eqfinp
\end{align*}
\medskip

\noindent $\bullet$ 
We prove Equation~\eqref{eq:level_set_l0_inter_sphere_a}:
  \begin{align*}
\SPHERE \cap \LevelSet{\lzero}{k} 
&= 
\defset{ \primal\in\RR^d }{ \norm{\primal} =1 \mtext{ and }
\lzero\np{\primal} \leq k }
\tag{ by definitions~\eqref{eq:Euclidian_SPHERE} of~$\SPHERE$ 
and \eqref{eq:level_curve_pseudonormlzero} of \( \LevelSet{\lzero}{k} \) }
\\
&= 
\defset{ \primal\in\RR^d }{ \norm{\primal} =1 \mtext{ and }
   \SymmetricGaugeNorm{k}{\primal} = \norm{\primal} }
\tag{by~\eqref{equiv_norm0}}
\\
&= 
\defset{ \primal\in\RR^d }{ \norm{\primal} =1 \mtext{ and }
   \SymmetricGaugeNorm{k}{\primal} = 1 }
\\
    &= 
    \SPHERE \cap \SymmetricGauge{k}{\SPHERE} 
\tag{ by definitions~\eqref{eq:Euclidian_SPHERE} and \eqref{eq:symmetric_gauge_norm_unit-sphere} of the spheres $\SPHERE$ 
and \( \SymmetricGauge{k}{\SPHERE} \) }
\\
    &= 
      \SPHERE \cap \SymmetricGauge{k}{\BALL} \cap
         \Bp{ \bigcup_{\cardinal{K} \leq k} \ESPHERE_{K} }
\tag{by property~\eqref{eq:symmetric_gauge_norm_unit-sphere} of the sphere
\( \SymmetricGauge{k}{\SPHERE} \) }
\\
    &= 
      \SPHERE \cap
         \Bp{ \bigcup_{\cardinal{K} \leq k} \ESPHERE_{K} }
\tag{as, by~\eqref{eq:k_approx_props}, 
we have that $\SPHERE \subset \SymmetricGauge{k}{\BALL}$ } 
\\
    &= \bigcup_{ {\cardinal{K} \leq k}} \bp{ \SPHERE \cap \ESPHERE_{K} }
\\
    &= 
    \bigcup_{ {\cardinal{K} \leq k}} \SPHERE_{K} 
\tag{ as \( \SPHERE \cap \ESPHERE_{K} = \SPHERE_{K} \) 
by~\eqref{eq:sphere_and_esphere_K} }
\eqfinp
\end{align*}
If we take over the proof where we use 
\(       \SymmetricGauge{k}{\SPHERE} 
= \SymmetricGauge{k}{\BALL} \cap
         \Bp{ \bigcup_{\cardinal{K} = k} \ESPHERE_{K} } \)
in~\eqref{eq:symmetric_gauge_norm_unit-sphere}, we obtain that
\( \SPHERE \cap \LevelSet{\lzero}{k} 
= \bigcup_{ {\cardinal{K} = k}} \SPHERE_{K} \). 
\medskip

\noindent $\bullet$ 
We prove Equation~\eqref{eq:level_set_l0_inter_sphere_b}.
First, we observe that the level set \( \LevelSet{\lzero}{k} \) is closed
because, by~\eqref{equiv_norm0}, it can be expressed as 
\( \LevelSet{\lzero}{k} = \defset{ \primal\in\RR^d }%
{ \SymmetricGaugeNorm{k}{\primal} = \norm{\primal} } \).
This also follows from the well-known property that 
the \pseudonormlzero~$\lzero$ is lower semi continuous.
Second, we have
\begin{align*}
  \SPHERE \cap \LevelSet{\lzero}{k} 
&= \SPHERE \cap 
\closedconvexhull\bp{\SPHERE \cap \LevelSet{\lzero}{k}} 
    \tag{by Lemma~\ref{lemma:convex_env} since 
\( \SPHERE \cap \LevelSet{\lzero}{k} \subset \SPHERE \) 
and is closed }
\\
&= \SPHERE \cap 
\closedconvexhull\bp{ \bigcup_{ {\cardinal{K} \leq k}} \SPHERE_{K} }
\tag{ as \( \SPHERE \cap \LevelSet{\lzero}{k} =
\bigcup_{ {\cardinal{K} \leq k}} \SPHERE_{K} \) 
by~\eqref{eq:level_set_l0_inter_sphere_a} }
\\
&= \SPHERE \cap \SupportBall{k}{\BALL} 
\tag{ as 
\( \closedconvexhull\bp{ \bigcup_{ {\cardinal{K} \leq k}} \SPHERE_{K} }
= \SupportBall{k}{\BALL} \) 
by~\eqref{eq:dual_support_norm_unit_ball} }
\eqfinp
\end{align*}
\medskip

\noindent $\bullet$ 
We prove Equation~\eqref{eq:closure_level_curve}.
For this purpose, we first establish the (known) fact that 
\( \overline{ \LevelCurve{\lzero}{k} } = \LevelSet{\lzero}{k} \). 
The inclusion \( \overline{ \LevelCurve{\lzero}{k} } 
\subset \LevelSet{\lzero}{k} \) is easy.
Indeed, as we have seen that \( \LevelSet{\lzero}{k} \) is closed, 
we have \( \LevelCurve{\lzero}{k} \subset \LevelSet{\lzero}{k} 
\Rightarrow
\overline{ \LevelCurve{\lzero}{k} } \subset 
\overline{ \LevelSet{\lzero}{k} } = \LevelSet{\lzero}{k} \).
There remains to prove the reverse inclusion
\( \LevelSet{\lzero}{k} \subset \overline{ \LevelCurve{\lzero}{k} } \).
For this purpose, we consider
\( \primal \in \LevelSet{\lzero}{k} \). 
If \( \primal \in \LevelCurve{\lzero}{k} \), obviously 
\( \primal \in \overline{ \LevelCurve{\lzero}{k} } \).
Therefore, we suppose that \( \lzero\np{\primal}=l < k \).
By definition of \( \lzero\np{\primal} \), there exists 
\( L \subset \ba{1,\ldots,d} \) such that 
\( \cardinal{L}=l < k \) and \( \primal = \primal_L \).
For \( \epsilon > 0 \), define \( \primal^\epsilon \) as
coinciding with  \( \primal \) except for 
$k-l$ indices outside~$L$ for which the components are 
\( \epsilon > 0 \).
By construction \( \lzero\np{\primal^\epsilon}=k \) and
\( \primal^\epsilon \to \primal \) when \( \epsilon \to 0 \).
This proves that 
\( \LevelSet{\lzero}{k} \subset \overline{ \LevelCurve{\lzero}{k} } \).

Second, we prove that \( \SPHERE \cap \LevelSet{\lzero}{k} 
= \overline{ \SPHERE \cap \LevelCurve{\lzero}{k} } \).
The inclusion 
\( \overline{ \SPHERE \cap \LevelCurve{\lzero}{k} } 
\subset \SPHERE \cap \LevelSet{\lzero}{k} \),
is easy.
Indeed, 
\( \overline{ \LevelCurve{\lzero}{k} } = \LevelSet{\lzero}{k} 
\Rightarrow
\overline{ \SPHERE \cap \LevelCurve{\lzero}{k} } \subset 
\overline{ \SPHERE } \cap \overline{ \LevelCurve{\lzero}{k} } = 
\SPHERE \cap \LevelSet{\lzero}{k} \).
To prove the reverse inclusion
\( \SPHERE \cap \LevelSet{\lzero}{k} 
\subset \overline{ \SPHERE \cap \LevelCurve{\lzero}{k} } \),
we consider \( \primal \in \SPHERE \cap \LevelSet{\lzero}{k} \).
As we have just seen that \( \LevelSet{\lzero}{k} = 
\overline{ \LevelCurve{\lzero}{k} }\), 
we deduce that \( \primal \in \overline{ \LevelCurve{\lzero}{k} }\).
Therefore, there exists a sequence
\( \sequence{z_n}{n\in\NN} \) in \( \LevelCurve{\lzero}{k} \)
such that \( z_n \to \primal \) when \( n \to +\infty \).
Since \( \primal \in \SPHERE \), we can always suppose that 
\( z_n \neq 0 \), for all $n\in\NN$. Therefore \( z_n/\norm{z_n} \) is well
defined and, when \( n \to +\infty \), 
we have \( z_n/\norm{z_n} \to \primal/\norm{\primal}=\primal \)
since \( \primal \in \SPHERE = \defset{\primal \in \PRIMAL}{\norm{\primal} = 1} \).
Now, on the one hand, 
\( z_n/\norm{z_n} \in \LevelCurve{\lzero}{k} \), for all $n\in\NN$,
and, on the other hand, \( z_n/\norm{z_n} \in \SPHERE \).
As a consequence \( z_n/\norm{z_n} \in \SPHERE \cap \LevelCurve{\lzero}{k} \),
and we conclude that \( \primal \in 
\overline{ \SPHERE \cap \LevelCurve{\lzero}{k} } \). 
Thus, we have proved that 
\( \SPHERE \cap \LevelSet{\lzero}{k} 
\subset \overline{ \SPHERE \cap \LevelCurve{\lzero}{k} } \).
\medskip

This ends the proof. 
\end{proof}

\begin{lemma} 
  If $A$ is a subset of the Euclidian sphere~$\SPHERE$ of~$\RR^d$, 
  then $A = \convexhull(A) \cap \SPHERE$. If $A$ is a closed subset of the Euclidian sphere~$\SPHERE$ of~$\RR^d$, 
  then $A = \closedconvexhull(A) \cap \SPHERE$.
  \label{lemma:convex_env}
\end{lemma}

\begin{proof} 
We first prove that $A = \convexhull(A) \cap \SPHERE$ when $A \subset \SPHERE$. 
Since $A \subset \convexhull(A)$ and $A\subset \SPHERE$, 
we immediately get that $A \subset\convexhull(A) \cap \SPHERE$.
To prove the reverse inclusion, we first start by proving that 
$\convexhull(A)\cap \SPHERE \subset \mathrm{extr}(\convexhull(A))$,
the set of extreme points of~$\convexhull(A)$.
  
The proof is by contradiction.
Suppose indeed that there exists $x\in \convexhull(A)\cap \SPHERE$ 
and $x\not\in \mathrm{extr}(\convexhull(A))$. Then, we could find 
$y \in \convexhull(A)$ and $z \in \convexhull(A)$, distinct from $x$,
and such that $x = \lambda y + (1-\lambda) z$ for some $\lambda\in (0,1)$.
Notice that necessarily \( y \neq z \) (because, else, we would have
$x=y=z$ which would contradict $y \neq x$ and $z \neq x$). 
By assumption $A\subset \SPHERE$, 
we deduce that $\convexhull(A) \subset \BALL= \defset{\primal \in
  \PRIMAL}{\norm{\primal} \leq 1}$, the unit ball, 
and therefore that \( \norm{y} \leq 1 \) and \( \norm{z} \leq 1 \).
If $y$ or $z$ were not in $\SPHERE$ --- that is, if either
\( \norm{y} < 1 \) or \( \norm{z} < 1 \) --- then we would obtain that
\( \norm{x} \leq \lambda \norm{y} + (1-\lambda) \norm{z} < 1 \) 
since 
$\lambda\in (0,1)$;
we would thus arrive at a contradiction since $x$ could not be in $\SPHERE$. 
Thus, both $y$ and $z$ must be in $\SPHERE$, 
and we have a contradiction since no $x \in \SPHERE$, 
the Euclidian sphere, can be obtained as a convex combination of 
$y \in \SPHERE$ and $z \in \SPHERE$, with \( y \neq z \).

Hence, we have proved by contradiction that 
$\convexhull(A)\cap \SPHERE \subset \mathrm{extr}(\convexhull(A))$.
We can conclude using the fact that 
$\mathrm{extr}(\convexhull(A)) \subset A$ 
(see \cite[Exercice 6.4]{hiriart1998optimisation}).

Now, we consider the case where the subset $A$ of the Euclidian sphere~$\SPHERE$
is closed.  Using the first part of the proof we have that
$A= \convexhull(A) \cap \SPHERE$.  Now, $A$ is closed by assumption and bounded
since $A\subset \SPHERE$. Thus, $A$ is compact and in a finite
dimensional space we have that $\convexhull(A)$ is compact~\cite[Th.~17.2]{Rockafellar:1970}, thus closed. We conclude that 
$A= \convexhull(A) \cap \SPHERE = \overline{\convexhull(A)} \cap \SPHERE = \closedconvexhull(A) \cap \SPHERE$, where the last equality comes from~\cite[Prop.~3.46]{Bauschke-Combettes:2017}.
\end{proof}

\newcommand{\noopsort}[1]{} \ifx\undefined\allcaps\def\allcaps#1{#1}\fi

\end{document}